\documentclass[11pt,leqno]{amsart}

\usepackage{hyperref}
\usepackage{mathrsfs}
\usepackage{amssymb} 
\usepackage{color}

\newtheorem{theorem}{Theorem}[section]

\newtheorem*{theorem*}{Theorem}
\newtheorem*{proposition*}{Proposition}
\newtheorem*{corollary*}{Corollary}
\newtheorem{lemma}[theorem]{Lemma}
\newtheorem{proposition}[theorem]{Proposition}
\newtheorem{corollary}[theorem]{Corollary}

\theoremstyle{definition}

\newtheorem{example}[theorem]{Example}

\theoremstyle{remark}
\newtheorem{remark}[theorem]{Remark}

\numberwithin{equation}{section}

\title[Maximal subgroups and irreducible representations]{Maximal
  subgroups and irreducible representations of generalised multi-edge
  spinal groups}

\author[B. Klopsch]{Benjamin Klopsch} \address{Benjamin Klopsch:
  Mathematisches Institut, Heinrich-Heine-Universit\"at, 40225
  D\"usseldorf, Germany} \email{klopsch@math.uni-duesseldorf.de}

\author[A. Thillaisundaram]{Anitha Thillaisundaram} 
\address{Anitha Thillaisundaram: School of Mathematics and Physics, University of Lincoln,
Lincoln LN6 7TS, UK}
\email{anitha.t@cantab.net}

\keywords{Tree automorphisms, branch groups, maximal subgroups,
  irreducible representations}

\subjclass[2010]{Primary 20E08;  Secondary 20E28, 20C07}

\begin{document}

\begin{abstract}
  Let $p \geq 3$ be a prime.  A generalised multi-edge spinal group
  \[
  G = \langle \{a \} \cup \{ b^{(j)}_i \mid 1 \leq j \leq p, \, 1 \leq
  i \leq r_j \} \rangle \leq \mathrm{Aut}(T)
  \]
  is a subgroup of the automorphism group of a regular $p$-adic rooted
  tree~$T$ that is generated by one rooted automorphism $a$ and $p$
  families $b^{(j)}_1, \ldots, b^{(j)}_{r_j}$ of directed
  automorphisms, each family sharing a common directed path disjoint
  from the paths of the other families.

  This notion generalises the concepts of multi-edge spinal groups,
  including the widely studied GGS-groups, and extended Gupta--Sidki
  groups that were introduced by Pervova.  Extending techniques that
  were developed in these more special cases, we prove: generalised
  multi-edge spinal groups that are torsion have no maximal subgroups
  of infinite index.  Furthermore we use tree enveloping algebras,
  which were introduced by Sidki and Bartholdi, to show that certain
  generalised multi-edge spinal groups admit faithful infinite
  dimensional irreducible representations over the prime
  field~$\mathbb{Z}/p\mathbb{Z}$.
\end{abstract}

\maketitle


\section{Introduction} \label{sec:intro}

Throughout let $p$ be an odd prime, and let $T$ denote a regular
$p$-adic rooted tree.  Pioneering constructions of Grigorchuk, Gupta
and Sidki in the 1980s led to the first examples of subgroups of the
automorphism group $\mathrm{Aut}(T)$ that are now called GGS-groups.
Since then the profinite group $\mathrm{Aut}(T)$ has become a
`building site' for finitely generated, residually finite groups with
interesting properties, in particular branch groups.  Typically the
groups are realised as subgroups of $\mathrm{Aut}(T)$ that are
generated by tree automorphisms with built-in self-similarities;
see~\cite{BarthGrigSunik,NewHorizons}.

In this paper we consider a collection $\mathscr{C}$ of subgroups
of~$\mathrm{Aut}(T)$, called \emph{generalised multi-edge spinal
  groups}, that form a common generalisation of multi-edge spinal
groups, studied in~\cite{Paper}, and extended Gupta--Sidki groups,
introduced in~\cite{Pervova1}.  Note that the class of multi-edge
spinal groups includes all GGS-groups.  With a certain amount of care
we extend the main results in~\cite{Paper,Pervova1,Sidki1, Vieira} to groups
in the larger class~$\mathscr{C}$.  The extended Gupta--Sidki groups,
originally manufactured by Pervova as examples of just infinite branch
groups without the congruence subgroup property, seem to have received
little attention beyond~\cite{Pervova1}.  It is reassuring that
results first obtained for multi-edge spinal groups carry over to
these groups.

For convenience, we give now an abridged definition of the class
$\mathscr{C}$ of generalised multi-edge spinal groups and illustrate
the concept with simple examples.  A more detailed discussion and the
definitions of standard terms can be found in
Section~\ref{sec:prelim}.

A generalised multi-edge spinal group
\begin{equation} \label{equ:def-gen-mult-ed-sp-gp} G = \big\langle \{a
  \} \cup \{ b^{(j)}_i \mid 1 \leq j \leq p, \, 1 \leq i \leq r_j \}
  \big\rangle
\end{equation}
is an infinite subgroup of (a Sylow-pro-$p$ subgroup of) the profinite
group $\mathrm{Aut}(T)$ that is generated by
\begin{itemize}
\item a rooted automorphism $a$ of order $p$ permuting cyclically the
  vertices $u_1,\ldots,u_p$ at the $1$st level of~$T$, and
\item families $\mathbf{b}^{(j)} = \{b^{(j)}_1, \ldots,
  b^{(j)}_{r_j}\}$, $j \in \{1,\ldots,p\}$, of directed automorphisms
  sharing a common directed path $P_j$ in~$T$.
\end{itemize}
The paths $P_1, \ldots, P_p$ are required to be mutually disjoint.
Without loss of generality we can demand that none of the generators
are superfluous, hence $0 \leq r_j \leq p-1$ for all $j \in
\{1,\ldots,p\}$.  Since $G$ is infinite, there is at least one $j \in
\{1,\ldots,p\}$ such that $r_j \not = 0$.

By construction such a generalised multi-edge spinal group is a
finitely generated, residually-(finite $p$) infinite group.  Regarded
as a subgroup of $\mathrm{Aut}(T)$ it is fractal, and under additional
assumptions, as we will see below, it is just infinite and branch.

\begin{example} \label{exa:extended-G-S}
  The \emph{extended Gupta--Sidki group}, or EGS-group for short, with
  \emph{defining vector} $\mathbf{e} = (e_1,\ldots,e_{p-1}) \in
  (\mathbb{Z}/p\mathbb{Z})^{p-1} \smallsetminus \{ \mathbf{0} \}$ is
  the group
  \[
  G = \langle a,b,c \rangle \leq \mathrm{Aut}(T),
  \]
  where the rooted automorphism $a$ permutes cyclically the $p$
  vertices at the $1$st level of~$T$, whereas the two directed
  automorphisms $b,c$ belong to the $1$st level stabiliser
  $\mathrm{Stab}_G(1)$ and satisfy the recursion relations
  \[
  b=(a^{e_1},\ldots,a^{e_{p-1}},b) \quad \text{and} \quad
  c=(c,a^{e_1},\ldots,a^{e_{p-1}}).
  \]
  Observe that each of the generators $a,b,c$ has order~$p$.
  In~\cite{Pervova1}, Pervova imposes two additional requirements:
  $\sum_{i=1}^{p-1} e_i = 0$ and $\mathbf{e}$ is
  non-symmetric, i.e., $e_i \ne e_{p-i}$ for some
  $i \in \{1,\ldots,p-1\}$.  The first condition is equivalent to $G$
  being torsion; see~\cite[Theorem~2]{Rozhkov} and
  \cite[Theorem~1]{Vovkivsky}.  Pervova shows under these assumptions:
  $G$ is just infinite and branch, but does not have the congruence
  subgroup property; see~\cite{Pervova1}. This is in
  contrast to the Grigorchuk group~\cite[Theorem~3.1]{MPI} and branch
  GGS-groups~\cite[Theorem~A]{GustavoAlejandraJone}.
\end{example}

\begin{example} \label{exa:multi-edge} Given $r \in \{1,\ldots,p-1\}$
  and a finite $r$-tuple $\mathbf{E}$ of
  $(\mathbb{Z}/p\mathbb{Z})$-linearly independent vectors
  \[
  \mathbf{e}_i = ( e_{i,1}, e_{i,2},\ldots , e_{i,p-1} ) \in
  (\mathbb{Z}/p\mathbb{Z})^{p-1}, \quad i\in \{1,\ldots,r \},
  \]
  we recursively define directed automorphisms $b_1, \ldots, b_r$ via
  \[
  b_i=(a^{e_{i,1}}, a^{e_{i,2}},\ldots,a^{e_{i,p-1}},b_i), \quad i\in
  \{1,\ldots,r \}.
  \]
  The group
  $G = \langle a, b_1, \ldots, b_r \rangle \leq \mathrm{Aut}(T)$ is
  the \emph{multi-edge spinal group} associated to the defining
  vectors $\mathbf{E}$.  We observe that $\langle a \rangle \cong C_p$
  and $\langle b_1, \ldots, b_r \rangle \cong C_p^{\, r}$ are
  elementary abelian $p$-groups.  These groups generalise GGS-groups,
  which correspond to the special case $r=1$.  In \cite{Paper} it was
  seen that the torsion multi-edge spinal groups are just infinite and
  branch.  Moreover, generalising results of Pervova on GGS-groups, it
  was shown there that torsion multi-edge spinal groups do not contain
  maximal subgroups of infinite index.  Equivalently, these groups do
  not contain proper dense subgroups with respect to the profinite
  topology.
\end{example}

 \begin{remark}\label{genGS} We will have occasion to look at
   \emph{generalised Gupta--Sidki groups}.  By this we mean GGS-groups
   $G = \langle a, b_1 \rangle$, i.e.\ $r=1$ in the notation above,
   with the extra property that
   $\{e_{1,1}, \ldots, e_{1,p-1}\} = \{1,2,\ldots,p-1\}$. This
   definition subsumes \emph{the} generalised Gupta--Sidki group
   $\langle a,b\rangle$, studied in \cite{Vieira}, where
   $b=(a,a^2,\ldots,a^{p-1},b)$.  For $p=3$ this group is the
   Gupta--Sidki  $3$-group.
\end{remark}

We extend the results illustrated by the two examples as follows.

\begin{theorem} \label{thm:main-result} Let $G$ be a generalised
  multi-edge spinal group, that is a group in the class~$\mathscr{C}$
  as described in \eqref{equ:def-gen-mult-ed-sp-gp}.
  \begin{enumerate}
  \item[(1)] If every non-empty family $\mathbf{b}^{(j)}$,
      $j \in \{1, \ldots, p\}$, features at least one non-constant
      defining vector, then $G$ is regular branch
      over $\gamma_3(G)$.

    Consequently, if $G$ is torsion then $G$ is just infinite and branch.
  \item[(2)] If the group $G$ is torsion then $G$ does not contain any
    proper dense subgroups, with respect to the profinite
    topology. The same holds for groups commensurable to $G$.
  \item[(3)] Suppose that the families $\mathbf{b}^{(1)}, \ldots,
    \mathbf{b}^{(p)}$ of directed generators of $G$ satisfy the
    additional conditions:
  \begin{enumerate}
  \item[(i)] every non-empty family $\mathbf{b}^{(j)}$,
    $j \in \{1, \ldots, p\}$, features at least one non-symmetric
    defining vector;
  \item[(ii)] there are at least two directed automorphisms, from two
    distinct families, that have the same defining vector.
  \end{enumerate}
  Then $G$ does not have the congruence subgroup property.
\end{enumerate}
\end{theorem}
   
For comparison, we remark that Bondarenko~\cite{Bondarenko} has shown
that there exist finitely generated branch groups that possess maximal
subgroups of infinite index. Bou-Rabee, Leemann and Nagnibeda
\cite{arxiv} proceeded to investigate weakly maximal subgroups of
branch groups, that is those that are maximal among the subgroups of
infinite index. The Grigorchuk group, the Gupta--Sidki group and many
other GGS-groups contain uncountably many non-parabolic weakly maximal
subgroups. In Corollary~\ref{weakly} we observe that these results
also apply to the branch groups in $\mathscr{C}$.

Returning to the results about maximal subgroups, as explained
in~\cite{Paper, Pervova4} they relate to a conjecture of
Passman~\cite[Conjecture~6.1]{Conjecture} on the group algebra $F[G]$
of a finitely generated group $G$ over a field $F$ of
characteristic~$p$.  The conjecture states that, if the Jacobson
radical $\mathrm{Jac}(F[G])$ coincides with the augmentation ideal
$\mathrm{Aug}(F[G])$, then $G$ is a finite $p$-group.
In \cite{Conjecture}, Passman proved that if
$\mathrm{Jac}(F[G]) = \mathrm{Aug}(F[G])$ then $G$ is a $p$-group and
every maximal subgroup of $G$ is normal of index~$p$.

The torsion groups in $\mathscr{C}$, having the property that maximal
subgroups are normal of index $p$, form a natural supply of potential
counter-examples to Passman's conjecture.  However, prominent examples
such as the Grigorchuk group and the generalised Gupta--Sidki
$p$-groups, do not satisfy the prerequisite
$\mathrm{Jac}(F[G]) = \mathrm{Aug}(F[G])$; see~\cite{Bartholdi,
  erratum, Sidki1, Vieira}.  In Section~\ref{sec:irr-reps}, we prove
similar results for a larger class of groups in $\mathscr{C}$.

The following question of
  Bergman~\cite[Problem~17.17]{Kourovka} was brought to our attention
  by A.~Abdollahi: do there exist finitely generated infinite groups
  with finitely many maximal subgroups.  As recorded in the Kourovka
  Notebook, an example of a $2$-generated $2$-group with $3$ maximal
  subgroups can be extracted from \cite[Sec. 7]{ErshovJaikin}.  Due to
  Theorem \ref{thm:main-result}(2) we obtain a large and more easily
  describable collection of finitely generated infinite groups with
  finitely many maximal subgroups.

Now a subgroup $G \leq \mathrm{Aut}(T)$ naturally acts on the boundary
$\partial T$ of the tree~$T$.  The tree enveloping algebra
$\mathfrak{A}_G$ of $G$ over the prime field
$F = \mathbb{Z}/p\mathbb{Z}$ is the image of $F[G]$ in the
endomorphism algebra
$\mathrm{End}(F \langle\!\langle \partial T \rangle\!\rangle)$ of the
$F$-vector space on the basis $\partial T$.  The image of
$\mathrm{Aug}(F[G])$ in $\mathfrak{A}_G$ is denoted by
$\mathrm{Aug}(\mathfrak{A}_G)$.  Bartholdi~\cite{Bartholdi, erratum}
extensively studied the tree enveloping algebra of the Grigorchuk
group, over various fields.  One of his key results is
that the tree enveloping algebra of the Grigorchuk group over a field
of characteristic~$2$ has a natural grading
\cite[Corollary~4.16]{Bartholdi}. This can be concluded from the
recursive presentation of the tree enveloping algebra
\cite[Theorem~4.15]{Bartholdi}. It remains open, as to whether the
tree enveloping algebra of the Gupta--Sidki groups, or of other
GGS-groups, similarly admit a natural grading. We see from the
following that it is desirable to close this gap in our knowledge.

\begin{theorem} \label{primitive} Let $G=\langle a, \mathbf{b}^{(1)}, \ldots, \mathbf{b}^{(p)} \rangle$ be a just infinite group
  in $\mathscr{C}$, and let $\mathfrak{A}_G$ be its tree enveloping
  algebra over $F = \mathbb{Z}/p\mathbb{Z}$. 

  If, either $G$ contains a generalised Gupta--Sidki group as a
  subgroup, or the induced augmentation ideal
  $\mathrm{Aug}(\mathfrak{A}_G)$ is a graded algebra with
    the elements $a-1$ and $b_i^{(j)}-1$ for $1\le j\le p$,
    $\, 1\le i\le r_j$ being homogeneous, then $G$ admits a faithful
  infinite-dimensional irreducible $F$-representation.
\end{theorem}

\begin{remark}
   The proof of Theorem~\ref{primitive} reveals that the
    groups $G$ to which it applies satisfy
    $\mathrm{Jac}(\mathfrak{A}_G) \ne \mathrm{Aug}(\mathfrak{A}_G)$.
    This implies that $\mathrm{Jac}(F[G])\ne \mathrm{Aug}(F[G])$, and
    hence, even though it may be torsion, $G$ cannot be a counter-example to
    Passman's Conjecture.
\end{remark}

Passman and Temple  showed in~\cite{PassmanTemple} that
for the  Gupta--Sidki $p$-group $\mathrm{GS}_p$ and $E$
  an algebraically closed field of characteristic $p$, if
  $E[\mathrm{GS}_p]$ has a non-trivial irreducible module, then
  $E[\mathrm{GS}_p]$ has infinitely many irreducible modules.  We
extend their result to branch groups in $\mathscr{C}$.

\begin{theorem} \label{thm:PassmanTemple}
  Let $G$ be a branch group in $\mathscr{C}$ and let $E$ be an
  algebraically closed field of characteristic $p$. If $E[G]$ has a
  non-trivial irreducible module, then $E[G]$ has infinitely many
  irreducible modules.
\end{theorem}


\section{Preliminaries} \label{sec:prelim}

 Let $T$ be the \textit{regular} $p$-adic rooted tree,
  meaning all vertices have the same out-degree~$p$.  Using the
  alphabet $X = \{1,2,\ldots,p\}$, the vertices $u_\omega$ of $T$ are
  labelled bijectively by elements $\omega$ of the free
  monoid~$\overline{X}$ in the following natural way.  The root of~$T$
  is labelled by the empty word~$\varnothing$, and for each word
  $\omega \in \overline{X}$ and letter $x \in X$ there is an edge
  connecting $u_\omega$ to~$u_{\omega x}$.  More generally, we say
  that $u_\omega$ precedes $u_\lambda$, or equivalently that $u_\lambda$
  succeeds $u_\omega$, whenever $\omega$ is a prefix of $\lambda$.

  There is a natural length function on~$\overline{X}$.  The words
  $\omega$ of length $\lvert \omega \rvert = n$, representing vertices
  $u_\omega$ that are at distance $n$ from the root, are the $n$th
  level vertices and constitute the \textit{$n$th layer} of the tree;
  the \textit{boundary} $\partial T$, whose elements correspond
  naturally to infinite rooted paths, is in one-to-one correspondence
  with the $p$-adic integers.

  Denote by $T_u$ the full rooted subtree of $T$ that has its root at
  a vertex~$u$ and includes all vertices succeeding~$u$.  For any
  two vertices $u = u_\omega$ and $v = u_\lambda$, the map
  $u_{\omega \tau} \mapsto u_{\lambda \tau}$, induced by replacing the
  prefix $\omega$ by $\lambda$, yields an isomorphism between the
  subtrees $T_u$ and~$T_v$.  We write $T_n$ to denote the subtree
  rooted at a generic vertex of level $n$.

  Every automorphism of $T$ fixes the root and the orbits of
  $\mathrm{Aut}(T)$ on the vertices of the tree $T$ are precisely its
  layers. For $f \in \mathrm{Aut}(T)$, the image of a vertex $u$ under
  $f$ is denoted by~$u^f$.  Observe that $f$ induces a faithful action
  on the monoid~$\overline{X}$ such that
  $(u_\omega)^f = u_{\omega^f}$.  For $\omega \in \overline{X}$ and
  $x \in X$ we have $(\omega x)^f = \omega^f x'$ where $x' \in X$ is
  uniquely determined by $\omega$ and~$f$.  This induces a permutation
  $f(\omega)$ of $X$ so that
  \[
  (\omega x)^f = \omega^f x^{f(\omega)}, \qquad \text{and consequently}
  \quad   (u_{\omega x})^f = u_{\omega^f x^{f(\omega)}}.
  \]
  The automorphism $f$ is \textit{rooted} if $f(\omega) = 1$ for
  $\omega \ne \varnothing$.  It is \textit{directed}, with directing
  path $\ell \in \partial T$, if the support
  $\{ \omega \mid f(\omega) \ne 1 \}$ of its labelling is infinite and
  marks only vertices at distance $1$ from the set of
    vertices corresponding  to the path~$\ell$.


\subsection{Subgroups of $\mathrm{Aut}(T)$}
Let $G$ be a subgroup of $\mathrm{Aut}(T)$ acting \textit{spherically
  transitively}, that is, transitively on every layer of $T$. The
\textit{vertex stabiliser} $\mathrm{Stab}_G(u)$ is the subgroup
consisting of elements in $G$ that fix the vertex~$u$.  For
$n \in \mathbb{N}$, the \textit{$n$th level stabiliser}
  $\mathrm{Stab}_G(n)= \bigcap_{\lvert \omega \rvert =n}
  \mathrm{Stab}_G(u_\omega)$
is the subgroup consisting of automorphisms that fix all vertices at
level~$n$.  Denoting by $T_{[n]}$ the finite subtree of $T$ on
vertices up to level~$n$, we see that $\mathrm{Stab}_G(n)$ is equal to
the kernel of the induced action of $G$ on $T_{[n]}$.

The full automorphism group $\mathrm{Aut}(T)$ is a profinite group:
\[
\mathrm{Aut}(T)= \varprojlim_{n\to\infty} \mathrm{Aut}(T_{[n]})
\]
The topology of $\mathrm{Aut}(T)$ is defined by the open subgroups
$\mathrm{Stab}_{\mathrm{Aut}(T)}(n)$, $n \in \mathbb{N}$.  The
subgroup $G$ of $\mathrm{Aut}(T)$ has the \textit{congruence subgroup
  property} if for every subgroup $H$ of finite index in $G$, there
exists some $n$ such that $\mathrm{Stab}_G(n)\subseteq H$. For branch
groups, having the congruence subgroup property is independent of the
choice of tree and action; see~\cite{Alejandra}. In fact,
  in most of the cases that we consider, there is essentially a unique
  tree and action associated to the group;
  cf.~Corollary~\ref{uniqueaction}.

Each $g\in \mathrm{Stab}_{\mathrm{Aut}(T)} (n)$ can be 
  described completely in terms of its restrictions to the subtrees
  rooted at vertices at level~$n$.  Indeed, there is a natural
  isomorphism
\[
\psi_n \colon \mathrm{Stab}_{\mathrm{Aut}(T)}(n) \rightarrow
\prod\nolimits_{\lvert \omega \rvert = n} \mathrm{Aut}(T_{u_\omega})
\cong \mathrm{Aut}(T) \times \overset{p^n}{\ldots} \times
\mathrm{Aut}(T).
\]

We write $U_u^G$ for the restriction of the vertex stabiliser
$\mathrm{Stab}_G(u)$ to the subtree $T_u$ rooted at a
vertex $u$. Since $G$ acts spherically transitively, the vertex
stabilisers at every level are conjugate under~$G$.  The common
isomorphism type of the restriction of the $n$th level vertex
stabilisers is the \textit{$n$th upper companion group} $U_n^G$ of
$G$.  The group $G$ is \textit{fractal} if every upper companion group
$U_n^G$ coincides with the group~$G$, after the natural identification
of subtrees.

The \textit{rigid vertex stabiliser} of $u$ in $G$ is the subgroup
$\mathrm{Rstab}_G(u)$ consisting of all automorphisms in $G$ that fix
all vertices $v$ of $T$ not succeeding~$u$.
  The \textit{rigid $n$th level stabiliser} is the product
  \[
  \mathrm{Rstab}_G(n) = \prod\nolimits_{\lvert \omega \rvert = n}
  \mathrm{Rstab}_G(u_\omega) \trianglelefteq G
  \]
  of the rigid vertex stabilisers of the vertices at level~$n$.  The
rigid vertex stabilisers at each level are conjugate under~$G$ and the
common isomorphism type $L_n^G$ of the $n$th level rigid vertex
stabilisers is called the \textit{$n$th lower companion group} of $G$.

We recall that the spherically transitive group $G$ is a \emph{branch
  group}, if $\mathrm{Rstab}_G(n)$ has finite index in $G$ for every
$n \in \mathbb{N}$.  For more detailed algebraic and geometric
characterisations see~\cite{NewHorizons}.  If, in addition, $G$ is
fractal and $1 \not = K \leq \mathrm{Stab}_G(1)$ with
$K\times \ldots \times K \subseteq \psi_1(K)$ and
$\lvert G : K \rvert < \infty$, then $G$ is said to be \emph{regular
  branch over $K$}.  Lastly we note that an infinite group $G$ is
\emph{just infinite} if all its proper quotients are finite.


\subsection{The collection $\mathscr{C}$ of generalised multi-edge
  spinal groups}

For $j \in \{ 1, \ldots, p \}$ let $r_j \in \{0,1,\ldots,p-1\}$, with
$r_j \not = 0$ for at least one index~$j$, and fix the numerical datum
$\mathbf{E} = (\mathbf{E}^{(1)}, \ldots, \mathbf{E}^{(p)})$, where
each
$\mathbf{E}^{(j)} = (\mathbf{e}^{(j)}_1, \ldots,
\mathbf{e}^{(j)}_{r_j})$
is an $r_j$-tuple of $(\mathbb{Z}/p\mathbb{Z})$-linearly independent
vectors
\[
\mathbf{e}^{(j)}_i = \big( e^{(j)}_{i,1}, \ldots, e^{(j)}_{i,p-1}
\big) \in (\mathbb{Z}/p\mathbb{Z})^{p-1}, \quad i \in \{1, \ldots, r_j
\}.
\]

By $a$ we denote the rooted automorphism, corresponding to the
$p$-cycle $(1 \; 2 \; \ldots \; p) \in \mathrm{Sym}(p)$, that
cyclically permutes the vertices at the first level of~$T$.  Observe that
\begin{equation} \label{equ:Sylow-pro-p} S = \big\{ f \in
  \mathrm{Aut}(T) \mid \forall \omega \in \overline{X} : f(\omega) \in
  \langle a \rangle \big\} \cong \varprojlim_{n \in \mathbb{N}} \,
  C_p \wr \ldots \wr C_p \wr C_p,
\end{equation}
the inverse limit of $n$-fold iterated wreath products of $C_p$,
forms a Sylow-pro\nobreakdash-$p$ subgroup of $\mathrm{Aut}(T)$.  The
\emph{generalised multi-edge spinal group} in \emph{standard form}
associated to the datum $\mathbf{E}$ is the group
\begin{align*}
  G = G_\mathbf{E} & = \langle a, \mathbf{b}^{(1)}, \ldots,
  \mathbf{b}^{(p)} \rangle \\
  & = \big\langle \{a \} \cup \{ b^{(j)}_i \mid 1 \leq j \leq p, \, 1
  \leq i \leq r_j \} \big\rangle \leq S,
\end{align*}
where, for each $j \in \{1,\ldots,p\}$, the generator family
$\mathbf{b}^{(j)} = \{b^{(j)}_1, \ldots, b^{(j)}_{r_j}\}$ consists of
commuting directed automorphisms $b^{(j)}_i \in \mathrm{Stab}_G(1)$
along the directed path
\[
\big( \varnothing, (p-j+1), (p-j+1)(p-j+1), \ldots \big) \in \partial T
\]
that satisfy the recursive relations
\[
\psi_1(b^{(j)}_i) = \Big( a^{e^{(j)}_{i,j}}, \ldots, a^{e^{(j)}_{i,p-1}},b^{(j)}_i,
a^{e^{(j)}_{i,1}},\ldots, a^{e^{(j)}_{i,j-1}} \Big);
\]
sometimes $\mathbf{e}^{(j)}_i$ is called the \emph{defining vector} of
$b^{(j)}_i$.  For each $j \in \{1,\ldots,p\}$ with $r_j \not = 0$ the
subgroup $\langle a, \mathbf{b}^{(j)} \rangle = \langle a, b^{(j)}_1,
\ldots, b^{(j)}_{r_j} \rangle$ of $G$ is a multi-edge spinal group;
compare Example~\ref{exa:multi-edge}.

Observe that the directing paths for the generator families
$\mathbf{b}^{(1)}, \ldots, \mathbf{b}^{(p)}$, i.e., the paths
\[
(\varnothing, p, pp, \ldots), \quad (\varnothing, (p-1), (p-1)(p-1),
\ldots), \quad \ldots, \quad (\varnothing, 1, 11, \ldots),
\]
are pairwise distinct.  We arrive at the notion of a \emph{generalised
  multi-edge spinal group}, given in abridged form in
Section~\ref{sec:intro}, by considering subgroups of $\mathrm{Aut}(T)$
that are conjugate to a generalised multi-edge spinal group in
standard form and declaring $\mathscr{C}$ to be the class of all such
groups.  Whenever they are branch, there is, in fact, a
unique branch action associated to these groups; see
  Corollary~\ref{uniqueaction}.


\section{First properties of generalised
  multi-edge spinal groups} \label{sec:properties}

\subsection{Basic properties} Here we include basic results concerning
generalised multi-edge spinal groups, i.e., groups in the class
$\mathscr{C}$.  Directly from the definition we deduce that each group
in $\mathscr{C}$ is infinite, fractal and acts spherically
transitively on~$T$.

\begin{lemma}\label{dagger}
  Let
  $G = \langle a, \mathbf{b}^{(1)}, \ldots, \mathbf{b}^{(p)} \rangle
  \in \mathscr{C}$
  be in standard form, and let $k \in \{1,\ldots, p\}$ with
  $r_k \not = 0$.  There exists an automorphism $f\in \mathrm{Aut}(T)$
  of the form $f=f_0f_1=f_1f_0$, where $f_0$ is a rooted automorphism
  corresponding to a permutation $\pi \in \mathrm{Sym}(p)$ with
  $(p-k+1)\pi = p-k+1$ and $f_1\in \mathrm{Stab}_G(1)$ with
  $\psi_1(f_1) = (f,\ldots,f)$, such that
  $G^f = \langle a, \tilde{\mathbf{b}}^{(1)}, \ldots,
  \tilde{\mathbf{b}}^{(p)} \rangle \in \mathscr{C}$
  is again in standard form and satisfies
  $\psi_1(\tilde{b}^{(k)}_1) = \big( a^{\tilde{e}^{(k)}_{1,k}},
  \ldots, a^{\tilde{e}^{(k)}_{1,p-1}},
  \tilde{b}^{(k)}_1,a,a^{\tilde{e}^{(k)}_{1,2}},\ldots,
  a^{\tilde{e}^{(k)}_{1,k-1}} \big)$.
\end{lemma}

\begin{proof}
  In essence, we may use the same proof as that of
  \cite[Lemma~3.3]{Paper}.
\end{proof}

Next we recall and adapt a reduction lemma from~\cite{Paper} to the
new situation.

\begin{lemma}[{\cite[Lemma~3.4]{Paper}}] \label{rowechelon} Let
  $G = \langle a, \mathbf{b}^{(1)}, \ldots, \mathbf{b}^{(p)} \rangle
  \in \mathscr{C}$
  be in standard form, and let $k \in \{1,\ldots, p\}$ with
  $r = r_k \not = 0$.  Then there exists a group
  $\tilde{G} = \langle a, \tilde{\mathbf{b}}^{(1)}, \ldots,
  \tilde{\mathbf{b}}^{(p)} \rangle \in \mathscr{C}$
  in standard form associated to the datum
  $\tilde{\mathbf{E}} = (\tilde{\mathbf{E}}^{(1)}, \ldots,
  \tilde{\mathbf{E}}^{(p)})$,
  with
  $\tilde{\mathbf{E}}^{(j)} = \big(\tilde {\mathbf{e}}^{(j)}_1,\ldots,
  \tilde {\mathbf{e}}^{(j)}_{r_j} \big)$
  supplying defining vectors for the directed automorphisms
    $\tilde{\mathbf{b}}^{(j)} = \big( \tilde{b}^{(j)}_1, \ldots,
    \tilde{b}^{(j)}_{r_j} \big)$,
  such that $\tilde G$ is conjugate to $G$ by an element
  $f \in \mathrm{Aut}(T)$ as in Lemma~\ref{dagger} and the following
  holds:
 \begin{enumerate}
 \item[(1)] $\tilde{e}^{(k)}_{i,1} = 1$ in $\mathbb{Z}/p\mathbb{Z}$ for each
   $i\in \{1,\ldots,r\}$;
 \item[(2)] if $r=2$ and $p=3$, then
   $\tilde{\mathbf{e}}^{(k)}_1=(1,0),
     \tilde{\mathbf{e}}^{(k)}_2=(1,1)$;
 \item[(3)] if $r=2$ and $p>3$, then either
   \begin{enumerate}
   \item[(a)] for each $i\in \{1,2\}$ there exists
     $m \in \{2,\ldots, p-2\}$ such that
     $\tilde{e}^{(k)}_{i,m-1} \tilde{e}^{(k)}_{i,m+1} \ne
     \big( \tilde{e}^{(k)}_{i,m} \big)^2$ in $\mathbb{Z}/p\mathbb{Z}$, or
 \item[(b)]
   $\tilde{\mathbf{e}}^{(k)}_1 = (1,0,\ldots,0,0),
     \tilde{\mathbf{e}}^{(k)}_2 = (1,0,\ldots,0,1)$;
   \end{enumerate}
 \item[(4)] if $r\ge 3$ then for each $i\in \{1,\ldots,r\}$ there
   exists $m \in \{2,\ldots, p-2\}$ such that
   $\tilde{e}^{(k)}_{i,m -1} \tilde{e}^{(k)}_{i,m+1} \ne \big(
   \tilde{e}^{(k)}_{i,m} \big)^2$ in $\mathbb{Z}/p\mathbb{Z}$.
\end{enumerate}
\end{lemma}

As in~\cite[Section~3]{Paper}, we identify some
  `exceptional' groups to be excluded from some of our
  results: let $\mathscr{C}_\mathrm{reg}$ be the class of
  groups that are conjugate in $\mathrm{Aut}(T)$ to a group
  $\langle a, \mathbf{b}^{(1)}, \ldots, \mathbf{b}^{(p)} \rangle \in
  \mathscr{C}$
  in standard form such that every non-empty generator family
  $\mathbf{b}^{(j)}$, $j \in \{1, \ldots, p\}$, features at least one
  non-constant defining vector $\mathbf{e}^{(j)}_i \not \in \{
  (\alpha,\ldots,\alpha) \mid \alpha \in \mathbb{Z}/p\mathbb{Z} \}$ with
  $1 \le i \le r_j$.

\begin{proposition} \label{sym} Let
  $G = \langle a, \mathbf{b}^{(1)}, \ldots, \mathbf{b}^{(p)} \rangle
  \in \mathscr{C}_\mathrm{reg}$ be in standard form. Then
  \[
  \psi_1(\gamma_3(\mathrm{Stab}_G(1))) = \gamma_3(G) \times
  \overset{p}{\ldots} \times \gamma_3(G).
  \]
  In particular,
  \[
  \gamma_3(G) \times \overset{p}{\ldots}\times \gamma_3(G) \subseteq
  \psi_1(\gamma_3(G)),
  \]
  and $G$ is regular branch over $\gamma_3(G)$.
\end{proposition}

\begin{proof} 
  By spherical transitivity, it suffices to show that
  \[
  \gamma_3(G) \times 1\times \ldots \times 1 \subseteq
  \psi_1(\gamma_3(\mathrm{Stab}_G(1))).
  \]
  Observe that $\gamma_3(G)$ is generated as a normal subgroup by
  commutators $[g_1,g_2,g_3]$ of elements $g_1, g_2, g_3$
  ranging over the generating set
  $\{a\} \cup \{ b^{(j)}_i \mid 1 \leq j \leq p, \, 1 \leq i \leq r_j
  \}$.
  For each $j \in \{1,\ldots,p\}$ with $r_j \not = 0$, the subgroup
  $G_j = \langle a, \mathbf{b}^{(j)} \rangle \leq G$ is a multi-edge
  spinal group, and \cite[Proposition~3.5]{Paper} shows that
  \begin{equation*}
    \gamma_3(G_j) \times \overset{p}{\ldots} \times \gamma_3(G_j)
    \subseteq \psi_1(\gamma_3(\mathrm{Stab}_{G_j}(1))) \subseteq
    \psi_1(\gamma_3(\mathrm{Stab}_G(1))).
  \end{equation*}
  Hence, it suffices to prove, for $k,l,m \in \{1,\ldots,p\}$ with
  $k \not = l$ and any given
  $c_j \in \{ b^{(j)}_1, \ldots, b^{(j)}_{r_j} \}$, $j \in \{k,l,m\}$,
  the elements
  \begin{equation} \label{equ:elements} ([a, c_k,c_l],1,\ldots,1) ,
    \quad ([c_k,c_l, a],1,\ldots, 1), \quad ([c_k,c_l,c_m],1,\ldots,1)
  \end{equation}
  are contained in $\psi_1(\gamma_3(\mathrm{Stab}_G(1)))$.
 
  First, we observe that
  \[
  ([c_k,c_l,c_m],1,\ldots,1) = \psi_1 \big( [c_k^{\, a^k}, c_l^{\,
    a^l}, c_m^{\, a^m}] \big) \in
  \psi_1(\gamma_3(\mathrm{Stab}_G(1))).
  \]
  Similarly, as
  $\psi_1([c_k^{\, a^k},c_l^{\, a^l}]) = ([c_k,c_l],1,\ldots,1)$, we
  can take $d \in \mathrm{Stab}_G(1)$ such that
  $\psi_1(d) = (a,*,\ldots, *)$, where the symbols $*$ denote
  unspecified elements, to deduce that
  \[
  ([c_k,c_l,a],1,\ldots,1) = \psi_1 \big( [c_k^{\, a^k}, c_l^{\,
    a^l},d] \big) \in \psi_1(\gamma_3(\mathrm{Stab}_G(1))).
  \]

  It remains to deal with the first type of commutator
  $([a,c_k,c_l],1,\ldots,1)$ listed in~\eqref{equ:elements}.  Working
  with a fixed $k \in \{1,\ldots, p\}$, but still allowing for
  modifications of the specific generators, we are free to conjugate
  by an element $f \in \mathrm{Aut}(T)$ as in Lemma~\ref{dagger}, and
  without loss of generality we may assume that the defining vectors
  $\mathbf{E}^{(k)}$ of the generators $\mathbf{b}^{(k)}$ have the
  form of $\tilde{\mathbf{E}}^{(k)}$ described in
  Lemma~\ref{rowechelon}.

  Let us specify $c_k = b^{(k)}_i$ and $c_l = b^{(l)}_j$, where
  $i \in \{1,\ldots, r_k\}$ and $j\in \{1, \ldots, r_l\}$.  Moreover,
  as we have already seen that
  $\gamma_3(G_k) \times 1 \times \ldots \times 1 \subseteq
  \psi_1(\gamma_3(\mathrm{Stab}_G(1)))$,
  it suffices to prove that there exists $x \in G_k$ such that
 \begin{multline*}
   \big( [a,b^{(k)}_i,b^{(l)}_j] [a,b^{(k)}_i,x]^{b^{(l)}_j}
   ,1,\ldots,1 \big) \\
   = \big( [a,b^{(k)}_i,xb^{(l)}_j],1,\ldots,1 \big) \in
   \psi_1(\gamma_3(\mathrm{Stab}_G(1))).
 \end{multline*}
 Arguing similar to~\cite[Proof of Proposition~3.5]{Paper}, we
 distinguish between three situations.

 \medskip

 \noindent \underbar{Case 1}: $r_k = 1$, and thus $i=1$.  Observing that
 \[
 h_j = \big( (b_1^{(k)})^{a^{k-2}} \big)^{-s}
 (b_j^{(l)})^{a^l}=(xb_j^{(l)},*,\ldots,*,1)
 \]
 for $s=e_{j,p-1}^{(l)}$ and $x\in G_k$, we obtain
 \[
 \big( \big[ a, b_1^{(k)}, xb_j^{(l)} \big], 1, \ldots, 1 \big) =
 \psi_1 \big( \big[(b_1^{(k)})^{a^{k-1}}, (b_1^{(k)})^{a^k}, h_j \big]
 \big) \in \psi_1(\gamma_3(\mathrm{Stab}_G(1))).
 \]

 \medskip

 \noindent \underbar{Case 2}: $r_k > 1$ and $(r_k,p) \ne (2,3)$.  By
 properties (3) and (4) of Lemma~\ref{rowechelon} there exists
 $m \in \{2,\ldots, p-2\} $ such that
 $e^{(k)}_{i,m -1} e^{(k)}_{i,m+1} \ne \big( e^{(k)}_{i,m}
 \big)^2$;
 apart from an exceptional case which only occurs for $r_k = 2$ to be
 dealt with below.  We set
 \[
 g_{i,m} = \Big( \big( b^{(k)}_i \big)^{a^{k-m}} \Big)^{e^{(k)}_{i,m}}
 \Big( \big(b^{(k)}_i \big)^{a^{k-m-1}} \Big)^{-e^{(k)}_{i,m-1}}
 \]
 which gives
 \[
 \psi_1(g_{i,m}) = \big( a^{\big( e^{(k)}_{i,m} \big)^2 -
   e^{(k)}_{i,m-1} e^{(k)}_{i,m+1}}, *, \ldots, *, 1 \big),
 \]
 where, in this case, the unspecified elements $*$ lie in $\langle a, b_i^{(k)}\rangle \le G_k$.  Since the first entry is non-trivial,
 there is a power $g_i$ of $g_{i,m}$ such that
 \[
 \psi_1(g_i) = (a,y,*,\ldots,*,1), \qquad \text{where $y \in G_k$.}
 \]
 Motivated by
 \[
 \psi_1 \Big( \big(b^{(l)}_j \big)^{a^l} \Big) = \big(
 b^{(l)}_j,*,\ldots,*, a^s \big), \qquad \text{where $s =
   e^{(l)}_{j,p-1}$,}
 \]
 and 
 \[\psi_1 \big( \big(g_i^{a^{-1}} \big)^{-s} \big) =
 (x,*,\ldots,*,1,a^{-s}) \qquad \text{where $x=y^{-s}$,} 
 \]
 we define $h_j = \big( g_i^{a^{-1}}
 \big)^{-s} \big( b^{(l)}_j \big)^{a^l}$ so that
 \[
 \psi_1(h_j) = \big( xb^{(l)}_j,*,\ldots,*,1 \big).
 \]

 As $e^{(k)}_{i,1} = 1$, it follows that
 \[
 \big( [a,b^{(k)}_i,x b^{(l)}_j],1,\ldots,1 \big) = \psi_1 \big(
 \big[(b^{(k)}_i)^{a^{k-1}},\big( b^{(k)}_i \big)^{a^k},h_j \big] \big) \in
 \psi_1(\gamma_3(\mathrm{Stab}_G(1)))
 \]
 and so we are done.

 It remains to deal with the exceptional case which occurs only for
 $r_k = 2$, and hence $p>3$.  According to property (3b) of
 Lemma~\ref{rowechelon},
 \[
 \mathbf{e}^{(k)}_1 = (1,0,\ldots,0), \quad
 \mathbf{e}^{(k)}_2=(1,0,\ldots,0,1),
 \]
 therefore
 \[
 \psi_1\big(b^{(k)}_1\big)^{a^{k-1}} = (a,1,\ldots,1,b^{(k)}_1), \quad
 \psi_1\big(b^{(k)}_2 \big)^{a^{k-1}}= (a,1,\ldots,1,a,b^{(k)}_2).
 \]
 In order to show that $([a,b^{(k)}_1,b^{(l)}_j],1,\ldots,1)$ and
 $([a,b^{(k)}_2,b^{(l)}_j],1,\ldots,1)$ are contained in
 $\psi_1(\gamma_3(\mathrm{Stab}_G(1)))$, it suffices to replace $g_i$
 in the generic argument given above by $(b^{(k)}_2)^{a^{k+1}}$ in
 both cases.

 In fact, for the first element we can also argue
 directly as follows: from the relation
 $\psi_1 \big( [(b^{(k)}_1)^{a^{k-1}},(b^{(k)}_1)^{a^k}] \big) =
 ([a,b^{(k)}_1],1,\ldots,1)$ we deduce
 \begin{multline*}
   \big( [a,b^{(k)}_1,b^{(l)}_j],1,\ldots,1 \big) = \psi_1 \big(
   \big[(b^{(k)}_1)^{a^{k-1}},\big( b^{(k)}_1 \big)^{a^k},
   \big(b^{(l)}_j\big)^{a^l}
   \big] \big) \\
   \in \psi_1(\gamma_3(\mathrm{Stab}_G(1))).
 \end{multline*}

 \medskip

 \noindent \underbar{Case 3}: $(r_k,p) = (2,3)$.  By property (2) of
 Lemma~\ref{rowechelon}, we may assume that
 $\mathbf{e}^{(k)}_1 = (1,0)$ and $\mathbf{e}^{(k)}_2 = (1,-1)$ so
 that
 \[
 \psi_1\big((b^{(k)}_1)^{a^{k-1}}\big) = (a,1,b^{(k)}_1) \quad
 \text{and} \quad \psi_1\big((b^{(k)}_2)^{a^{k-1}}\big) =
 (a,a^{-1},b^{(k)}_2).
 \]
 Setting
   $h = \big( (b_1^{(k)})^{a^{k-2}} \big)^{-s} \big( b_j^{(l)}
   \big)^{a^l}$
   for $s=e^{(l)}_{j,p-1}$, we obtain $\psi_1(h) = (b_j^{(l)},*,1) $
   and
 \begin{align*}
   \big( \big[a, b_1^{(k)}, b_j^{(l)} \big], 1,  1 \big) & = \psi_1
                                                                  \big(
                                                                  \big[ (b_1^{(k)})^{a^{k-1}},
                                                                  (b_1^{(k)})^{a^k},
                                                                  (b_j^{(l)})^{a^l}
                                                                  \big]
                                                                  \big),
   \\ 
   \big( \big[a, b_2^{(k)}, b_j^{(l)} \big],1,1 \big) & =
                                                               \psi_1 \big( \big[ (b_1^{(k)})^{a^{k-1}},
                                                               (b_2^{(k)})^{a^k},
                                                               h \big]
                                                               \big) 
 \end{align*}
 so that both elements lie in $\psi_1(\gamma_3(\mathrm{Stab}_G(1)))$.
\end{proof}

Next we record the following result regarding the derived
subgroup $G'$ of~$G$, based on the extra assumption that there are
sufficiently many non-symmetric defining vectors.

\begin{proposition} \label{derived} Let
  $G = \langle a, \mathbf{b}^{(1)},\ldots , \mathbf{b}^{(p)}\rangle
  \in \mathscr{C}$
  be in standard form and such that every non-empty family
  $\mathbf{b}^{(j)}$, $j \in \{1, \ldots, p\}$, features at least one
  non-symmetric defining vector.  Then
  \[
  \psi_1(\textup{Stab}_G(1)')=G'\times \overset{p}{\ldots}\times G'.
  \]
  In particular,
  $G'\times \overset{p}{\ldots}\times G' \subseteq \psi_1(G')$,
  and  $G$ is regular branch over $G'$.
\end{proposition}

\begin{proof}
  By spherical transitivity, it suffices to show that 
  \[
  G' \times 1\times \ldots \times 1 \subseteq
  \psi_1(\mathrm{Stab}_G(1)').
  \]
  Observe that $G'$ is generated as a normal subgroup by commutators
  $[g_1,g_2]$, where $g_1, g_2$ range over the generating
  set
  $\{a\} \cup \{ b^{(j)}_i \mid 1 \leq j \leq p, \, 1 \leq i \leq r_j
  \}$.

  For each $j \in \{1,\ldots,p\}$ with $r_j \not = 0$, we consider the
  associated multi-edge spinal group
  $G_{j} = \langle a, \mathbf{b}^{(j)}\rangle$.  Without loss of
  generality we may assume that all defining vectors for the family
  $\mathbf{b}^{(j)}$ are non-symmetric.  The $r_j$ subgroups
  \[
  H_j(1) = \langle a, b_1^{(j)} \rangle, \quad \ldots\quad , \quad H_j(r_j)
  = \langle a, b_{r_j}^{(j)} \rangle
  \]
  of $G_{j}$ are GGS-groups and satisfy the corresponding statement to
  our claim; see~\cite[Lemma~3.4]{Amaia}.  Consequently,
 \[
 ([a,b_i^{(j)}], 1, \ldots,  1) \in \psi_1(\mathrm{Stab}_{H_j(i)}(1)')
 \subseteq \psi_1(\mathrm{Stab}_G(1)') \quad \text{for $1\le i\le r_j$.}
 \]

 Hence it suffices to observe that, for $k,l \in \{1,\ldots,p\}$ with
 $k \not = l$ and any given
 $c_j \in \{ b^{(j)}_1, \ldots, b^{(j)}_{r_j} \}$, $j \in \{k,l\}$,
 the element
 $([c_k,c_l],1,\ldots,1) = \psi_1([c_k^{\, a^{k}}, c_l^{\, a^{l}}])$ is
 contained in $\psi_1(\mathrm{Stab}_G(1)')$.
\end{proof}

\begin{corollary}
  The groups in $\mathscr{C}_\mathrm{reg}$ are branch,
  and the torsion groups in $\mathscr{C}$ are just
  infinite.
\end{corollary}

\begin{proof}
  From Proposition~\ref{sym}, it follows similar to
  \cite[Propositions~3.6 and~3.7]{Paper} that the groups in
  $\mathscr{C}_\mathrm{reg}$ are branch.  

  As indicated in Example~\ref{exa:extended-G-S}, torsion groups in
  $\mathscr{C}$ are already in $\mathscr{C}_\mathrm{reg}$;
  see~\cite[Theorem~2]{Rozhkov} and \cite[Theorem~1]{Vovkivsky}.
  Finally, finitely generated torsion branch groups are just infinite;
  see~\cite[Section~7]{NewHorizons}.
\end{proof}

Hence we have established part (1) of Theorem~\ref{thm:main-result}.

\bigskip

We end this section by proving that the branch groups in $\mathscr{C}$
have an essentially unique `branch action'. For
  vertices $u,v$ of $T$ we write $u \le v$ if $u$ precedes~$v$, and
  $u < v$ for $u \le v$, but $u \ne v$.  In~\cite{GrigWils2},
  Grigorchuk and Wilson (using a different notational convention)
  introduced the following condition on a branch group $G$ acting on a
  regular $p$-adic tree:
\begin{itemize}
\item[$(\dagger)$] Whenever $u,u',v$ are vertices of $T$ such that
  $u,u'$ are incomparable and $u<v$, there exists $g\in G$ such that
  $(u')^g = u'$, but~$v^g \not = v$.
\end{itemize}

With this we state \cite[Theorem~1]{GrigWils2}, for branch groups
acting on regular $p$-adic rooted trees.

\begin{theorem}[Grigorchuk, Wilson] \label{Theorem0.1} Let $G$ be a
  branch group acting on a regular $p$-adic rooted tree $T$ and
  suppose that~$(\dagger)$ holds.  Let $T'$ be any other spherically
  homogeneous rooted tree on which $G$ acts as a branch group.  Then
  there is a $G$-equivariant isomorphism from $T'$ to a tree obtained
  from $T$ by deletion of layers.
\end{theorem}

This motivates us to prove the following result on GGS-groups.

\begin{proposition} \label{condition**} Let $G=\langle a,b \rangle$ be
  a GGS-group, acting on a regular $p$-adic tree $T$ for~$p\ge 5$.
  Suppose further that the defining vector $(e_1,\ldots, e_{p-1})$ of
  the directed automorphism $b$ is non-constant and that all its
  entries are non-zero.  Then $G$ satisfies condition~$(\dagger)$.
\end{proposition}

\begin{proof}
  Let $u,u'$ be incomparable vertices of~$T$.  We denote by
  $\bar{u},\bar{u}'$ the first level vertices satisfying
  $\bar{u} \le u$ and $\bar{u}' \le u'$.  Suppose that
    $\bar{u} = u_i$ and $\bar{u}' = u_j$ for labels
    $i,j \in \{1,\ldots,p\}$.  Applying conjugation by a suitable
  power of~$a$, if necessary, we may assume without loss of generality
  that $i< j$.  Replacing $b$ by a suitable power of itself, if
  necessary, we may further assume that~$e_1=1$.

  \smallskip

  \noindent
  \underline{Case 1}: There exists  $m\ne 1$ such that
  \[
  \psi_1(b)=(a,a^m,a^{m^2},\ldots,a^{m^{p-2}},b).
  \]
  This implies
  \[
  \psi_1(b(b^{-m})^a)=(ab^{-m},1,\ldots,1,ba^{-1}).
  \]

  First suppose that $(i,j)\ne (1,p)$.  Then the element
  $g = (b(b^{-m})^a)^{a^{i-1}} \in \mathrm{Stab}_G(1)$ has the form
  \[
  \psi_1(g) = (1,\ldots,1,ba^{-1},
  \underbrace{ab^{-m}}_{i^{\text{th}} \text{ entry}}, 1,\ldots,1 ).
  \]
  Consequently, $(u')^g = u'$, but no vertex $v$ with
    $v>\bar{u}$ is fixed by $g$ due to the factor $a$ in the $i$th
  entry of~$\psi_1(g)$.  Therefore condition $(\dagger)$ is satisfied.

  In the remaining case $(i,j)=(1,p)$, we instead utilise
  $g = (b(b^{-m})^a)^{a}$ where
  \[
  \psi_1(g)=(ba^{-1},ab^{-m},1,\ldots,1).
  \]

  \smallskip

  \noindent
  \underline{Case 2}: There exists $k \in \{2,\ldots, p-2\}$ such that
  \begin{equation} \label{equ:k-minimal} \frac{e_2}{e_1} = \ldots =
    \frac{e_k}{e_{k-1}} \quad \text{ but } \quad \frac{e_k}{e_{k-1}}
    \ne \frac{e_{k+1}}{e_k}.
  \end{equation}
  Then, as in Case~2 of the proof of Proposition~\ref{sym},  we define
  \[
  g_k = \big( b^{a^{p-k+1}} \big)^{e_k} \big( b^{a^{p-k}}
  \big)^{-e_{k-1}} \in \mathrm{Stab}_G(1)
  \]
  so that
  \[
  \psi_1(g_k)=(a^{f_1},\ldots, a^{f_{p-k-1}},
  a^{f_{p-k}}b^{-e_{k-1}},b^{e_k}a^{f_{p-k+1}},1,\overset{k-1}{\ldots},1),
  \]
  where 
  \begin{align*}
    f_1 & =e_k^2-e_{k-1}e_{k+1} \ne 0, \\
    f_l & = e_ke_{k+l-1 }- e_{k-1}e_{k+l} \quad \text{ for $2 \le l \le
          p-k-1$,} \\
    f_{p-k} & =e_ke_{p-1} \ne 0 \qquad \text{and} \qquad 
              f_{p-k+1}=-e_{k-1} \ne 0.
  \end{align*}
  The first statement follows from~\eqref{equ:k-minimal}, while the
  last two are due to the circumstance that
  $e_1, \ldots, e_{p-1} \ne 0$.

  We now identify a suitable conjugate $g$ of $g_k$.  First we take
  any conjugate $h$ of $g_k$ by a power of $a$ such that the $j$th
  entry of $\psi_1(h)$ does not involve a non-trivial power
  of~$a$. Hence the $j$th entry is trivial by the form of
  $\psi_1(g_k)$.  If the $i$th entry of $\psi_1(h)$ also does not
  involve a non-trivial power of~$a$, then consider $h^{a^l}$, where
  $l=j-i$.  Now the $j$th entry of $\psi_1(h^{a^l})$ is trivial.  If
  the $i$th entry of $\psi_1(h^{a^l})$ does not involve a non-trivial
  power of~$a$, then repeat the process and consider $h^{a^{2l}}$,
  etc.  Since there are entries of $h$ that involve non-trivial powers
  of~$a$, we arrive in this way at $g = h^{a^{ml}}$, for some
  $m\in \{0, 1,\ldots, p-1\}$, such that the $j$th entry of
  $\psi_1(g)$ is trivial, while the $i$th entry involves a non-trivial
  power of~$a$. Thus condition $(\dagger)$ is again satisfied.
\end{proof}

We remark that GGS-groups $G = \langle a,b \rangle$ are branch apart
from when the defining vector of the directed automorphism $b$ is
constant.  In the latter case the group $G$ is only weakly branch,
that is, all rigid stabilisers are non-trivial;
  see~\cite[Theorem~3.7]{GustavoAlejandraJone}
  and~\cite[Lemma~4.2(iii)]{Amaia}.  The next corollary extends the
results in~\cite{GrigWils2} which cover GGS-groups with defining
vector having at least one zero entry, and additionally the
Gupta--Sidki $3$-group.

\begin{corollary} \label{uniqueaction} Let $G\in \mathscr{C}$ be
  branch.  Then the branch action of $G$ is unique in the sense of
  Theorem~\ref{Theorem0.1}.
\end{corollary}

\begin{proof}
  Observe that $G$ contains a branch GGS-group.  Thus
  Proposition~\ref{condition**} and \cite[Lemmata~6 and~7]{GrigWils2}
  imply that $G$, too, satisfies condition $(\dagger)$.  Hence the result
  follows by Theorem~\ref{Theorem0.1}.
\end{proof}


\subsection{Length functions and abelianisation} For the proof of part
(2) of Theorem~\ref{thm:main-result} we require certain length
functions on the groups $G \in \mathscr{C}$; as a by-product we pin
down the abelianisation $G/G'$ of~$G$. Fix a group
$G= \langle a,\mathbf{b}^{(1)}, \ldots, \mathbf{b}^{(p)} \rangle \in
\mathscr{C}$ in standard form and consider the free product
\[
\widehat{G} = \langle \hat a\rangle * \langle
\widehat{\mathbf{b}}^{(1)} \rangle * \ldots * \langle
\widehat{\mathbf{b}}^{(p)} \rangle
\]
of elementary abelian $p$-groups $\langle \hat a \rangle \cong C_p$
and
$\langle \widehat{\mathbf{b}}^{(j)} \rangle = \langle \hat b_1^{(j)},
\ldots, b_{r_j}^{(j)} \rangle \cong C_p^{\, r_j}$
for $1 \leq j \leq p$.  Note that there is a unique epimorphism
$\pi \colon \widehat{G} \rightarrow G$ such that $\hat a \mapsto a$
and $\hat b_i^{(j)} \mapsto b_i^{(j)}$ for $1 \leq j \leq p$ and
$1 \leq i \leq r_j$, inducing an epimorphism from
$\widehat{G}/\widehat{G}' \cong C_p^{\, 1 + r_1 + \ldots + r_p}$
onto~$G/G'$.  We want to show that the latter is an isomorphism; see
Proposition~\ref{pro:abelianisation}.
 
By the standard theory of free products of groups, each element
$\hat{g} \in \widehat{G}$ has a unique reduced form
\[
\hat{g} = \hat a^{\alpha_1}\, w_1 \, \hat a^{\alpha_2} \, w_2 \,
\cdots \, \hat a^{\alpha_l} \, w_l \, \hat a^{\alpha_{l+1}},
\]
where $l \in \mathbb{N} \cup \{0\}$,
$w_1,\ldots, w_l \in \langle \widehat{\mathbf{b}}^{(1)} \cup \ldots
\cup \widehat{\mathbf{b}}^{(p)} \rangle \smallsetminus \{1\}$,
and $\alpha_1, \ldots, \alpha_{l+1} \in \mathbb{Z}/p\mathbb{Z}$ such
that $\alpha_i \ne 0$ for $i \in \{2,\ldots,l\}$.
Furthermore, for each $i\in \{1,\ldots, l\}$, the element $w_i$ can be
uniquely expressed as
\[
w_i =  \big( \widehat{\mathbf{b}}^{(k(i,1))} \big)^{\boldsymbol{\beta}(i,1)}
\cdots \big( \widehat{\mathbf{b}}^{(k(i,n_i))} \big)^{\boldsymbol{\beta}(i,n_i)},
\]
where $n_i \in \mathbb{N}$,
$k(i,1), \ldots, k(i,n_i) \in \{1,\ldots, p\}$, with
$k(i,m) \ne k(i,m +1)$ for $1 \le m \le n_i-1$, and
\[
\boldsymbol{\beta}(i,m) = \big( \beta(i,m)_1, \ldots,
\beta(i,m)_{r_{k(i,m)}} \big) \in
(\mathbb{Z}/p\mathbb{Z})^{r_{k(i,m)}} \smallsetminus \{\mathbf{0}\}, \;\;
1 \le m \le n_i,
\]
are exponent vectors so that
\[
\big( \widehat{\mathbf{b}}^{(k(i,m))} \big)^{\boldsymbol{\beta}(i,m)} = \big(
\hat b^{(k(i,m))}_1 \big)^{\beta(i,m)_1} \cdots \big(
\hat b^{(k(i,m))}_{r_{k(i,m)}} \big)^{\beta(i,m)_{r_{k(i,m)}}}.
\]
The \emph{length} of $\hat{g}$ is defined as
\[
\partial(\hat{g}) =  n_1 + \ldots + n_l.
\]

Furthermore, we define \emph{exponent maps} from $\widehat{G}$ to $\mathbb{Z}/p\mathbb{Z}$ by
\begin{align*}
  \varepsilon_{\hat a}(\hat{g}) & = \sum_{m=1}^{l+1} \alpha_m \quad \text{and} \\
  \varepsilon_{\hat b^{(k)}_j}(\hat{g}) & = \sum_{\substack{1 \le i \le l, \, 1
                                          \le m \le n_i \\ \text{s.t. } k(i,m) = k}} \beta(i,m)_j \quad
  \text{for $1 \le k \le p$ and $1 \le j \le  r_k$.} 
\end{align*}
 
The isomorphism $G/[G,G]\cong \widehat{G}/[\widehat{G},\widehat{G}]$
is obtained parallel to \cite[Section~4.1]{Paper}, which uses a
reformulation of Rozhkov~\cite{Rozhkov}.

\begin{proposition} \label{pro:abelianisation} Let $G\in\mathscr{C}$
  be in standard form and $\widehat{G}$ as above.  Then the surjective
  homomorphism
  \[
  \widehat{G} \rightarrow (\mathbb{Z}/p\mathbb{Z}) \times
  \prod_{j=1}^p (\mathbb{Z}/p\mathbb{Z})^{r_j}, \quad \hat{g} \mapsto
  \big( \varepsilon_{\hat a}(\hat{g}), \big( (\varepsilon_{\hat
    b^{(1)}_i}(\hat{g}))_{i=1}^{r_1}, \ldots, (\varepsilon_{\hat
    b^{(p)}_i}(\hat{g}))_{i=1}^{r_p} \big) \big)
  \]
  with kernel $\widehat{G}'$ factors through $G/G'$ and consequently,
  \begin{align*}
    G/G' & \cong \langle \underline{a},  \, \underline{b}^{(1)}_1 , \ldots,
           \underline{b}^{(1)}_{r_1}, \quad \ldots \phantom{x} \ldots,
           \quad \underline{b}^{(p)}_1, 
           \ldots, \underline{b}^{(p)}_{r_p} \rangle \\ 
         & \cong C_p \times \overset{1+r_1+\ldots +r_p}{\ldots} \times C_p. 
  \end{align*}
\end{proposition}

Let $G \in \mathscr{C}$ and $\pi \colon \widehat{G} \rightarrow G$ be
the natural epimorphism as above. The \emph{length} of $g\in G$ is
\[
\partial(g) = \min \{\partial(\hat{g}) \mid \hat{g}\in \pi^{-1}(g) \}.
\]
Moreover, via Proposition~\ref{pro:abelianisation}, we define
\[
\varepsilon_a(g) = \varepsilon_{\hat a}(\hat{g}) \quad \text{and}
\quad \varepsilon_{b^{(j)}_i}(g) = \varepsilon_{\hat
  b^{(j)}_i}(\hat{g}), \, \text{for $1 \le j \le p$ and
  $1 \le i \le r_j$,}
\]
via any pre-image $\hat{g}\in \pi^{-1}(g)$.

Since $G$ is fractal, every $g \in G$ may be expressed
as
\[
g = \psi_1^{-1}(g_1,\ldots , g_p) \, a^{\varepsilon_a(g)},
\]
where $g_i\in G$ for $1\le i \le p$.  Of course, the decomposition can be
applied repeatedly, yielding, for instance,
$g_i = \psi_1^{-1}(g_{i,1},\ldots, g_{i,p}) \, a^{\varepsilon_a(g_i)}$ for
$1\le i \le p$.

\begin{lemma} \label{lemma4.1} Let
  $G = \langle a, \mathbf{b}^{(1)},\ldots , \mathbf{b}^{(p)}\rangle\in
  \mathscr{C}$
  be in standard form, and let $g\in G$.  Then, using the notation
  introduced above, $\sum_{i=1}^p \partial(g_i) \le \partial(g)$.

  Suppose further that $\partial(g)>1$. Then
  $\partial(g_{i,j}) < \partial(g)$ for all $i,j\in \{1,\ldots,p\}$.
\end{lemma}

\begin{proof}
  Let $\partial(g)=m$.  We may express
  $g a^{-\varepsilon_a(g)} = \psi_1^{-1}(g_1,\ldots,g_p)$ as
 \[
 \big( (\mathbf{b}^{(k(1))})^{\boldsymbol{\beta}(1)} \big)^{a^{e_1}}
 \big( (\mathbf{b}^{(k(2)} )^{\boldsymbol{\beta}(2)} \big)^{a^{e_2}}
 \cdots \big( (\mathbf{b}^{(k(m))})^{\boldsymbol{\beta}(m)} \big)^{a^{e_m}} 
 \]
 with $k(i) \in \{1, \ldots, p\}$, $e_i\in \mathbb{Z}/p\mathbb{Z}$ and
 exponent vectors
 $\boldsymbol{\beta}(i) \in (\mathbb{Z}/p\mathbb{Z})^{r_{k(i)}}
 \setminus \{ \mathbf{0} \}$
 for $1 \le i \le m$.  Furthermore, we have $k(i) \ne k(i+1)$ whenever
 $e_i = e_{i+1}$.  For each $i \in \{1,\ldots, m\}$, the factor
 $\big( (\mathbf{b}^{(k(i))})^{\boldsymbol{\beta}(i)} \big)^{a^{e_i}}$
 contributes to precisely one coordinate $g_{j(i)}$ a factor
 $(\mathbf{b}^{(k(i))})^{\boldsymbol{\beta}(i)}$ and to all other
 coordinates $g_l$, $l \not = j(i)$, a power of $a$.  Hence the first
 statement of the lemma follows.

 Suppose that $\partial(g_{i,j}) = \partial(g)$ for some $i,j \in \{1,
 \ldots, p \}$.  This implies $\partial(g_i) = \partial(g)$ and
 \[
 g_i = \big(\mathbf{b}^{(k(1))} \big)^{\boldsymbol{\beta}(1)} \big(
 \mathbf{b}^{(k(2))} \big)^{\boldsymbol{\beta}(2)} \cdots
 \big(\mathbf{b}^{(k(m))} \big)^{\boldsymbol{\beta}(m)}.
 \]
 Now either $k(1) = \ldots = k(m)$ and then
 $\partial(g) = \partial(g_i) \in \{0, 1\}$, or there exists
 $l\in \{1,\ldots, m-1\}$ such that $k(l) \ne k(l+1)$ and thus
 $\partial(g_i)>\partial(g_{i,j})$ for $1\le j\le p$.
\end{proof}

We now prove part (3) of Theorem~\ref{thm:main-result}.

\begin{proposition} \label{CSP} Let
  $G= \langle a,\mathbf{b}^{(1)}, \ldots, \mathbf{b}^{(p)} \rangle \in
  \mathscr{C}$
  be in standard form and such that the families
  $\mathbf{b}^{(1)}, \ldots, \mathbf{b}^{(p)}$ of directed generators
  of $G$ satisfy the additional conditions
  \begin{enumerate}
  \item[(i)] every non-empty family $\mathbf{b}^{(j)}$,
    $j \in \{1, \ldots, p\}$, features at least one non-symmetric
    defining vector;
  \item[(ii)] there are at least two directed automorphisms, from two
    distinct families, that have the same defining vector.
  \end{enumerate}
  Then $G$ does not have the congruence subgroup property.
\end{proposition}

\begin{proof}
 Based on condition (ii), we find directed generators
  $b \in \mathbf{b}^{(i)}$ and $c \in \mathbf{b}^{(j)}$, where
  $1 \le i < j \le p$, and
  $(e_1,\ldots,e_{p-1})\in (\mathbb{Z}/p\mathbb{Z})^{p-1}$ such that
  \[
  b = (a^{e_i},\ldots, a^{e_{p-1}},b,a^{e_1},\ldots, a^{e_{i-1}}),
  \quad  c = (a^{e_j},\ldots, a^{e_{p-1}},c,a^{e_1},\ldots, a^{e_{j-1}}).
  \] 
  We proceed as in~\cite[Lemma~3.1 and Corollary~3.1]{Pervova1}.
 
  First we construct recursively, for $n\in \mathbb{N}$, elements
  $t_n\in bG' \cap c \, \mathrm{Stab}_G(n)$.  Set $t_1=b$.  For
  $n \ge 2$, suppose
  $t_{n-1} \in b G' \cap c \, \mathrm{Stab}_G(n-1)$.  Condition (i)
  tells us that Proposition~\ref{derived} is at our disposal so that
  \[
  x_n
  :=\psi_1^{-1}(1,\overset{p-j}{\ldots},1,b^{-1}t_{n-1},1,\overset{j-1}{\ldots},1)
  \in G'.
  \]
  Setting $t_n = b^{a^{i-j}} x_n \in b G'$, we conclude that
  \[
  \psi_1(t_n) = (a^{e_j},\ldots, a^{e_{p-1}},t_{n-1},a^{e_1},\ldots,
  a^{e_{j-1}})
  \]
  and thus
  \[
  \psi_1(c^{-1} t_n) =
  (1,\overset{p-j}{\ldots},1,c^{-1}t_{n-1},1,\overset{j-1}{\ldots},1).
  \]
  Since $c^{-1} t_{n-1} \in \mathrm{Stab}_G(n-1)$, we see that
  $c^{-1} t_n \in \mathrm{Stab}_G(n)$, and hence
  $t_n\in bG' \cap c \, \mathrm{Stab}_G(n)$.

  To finish, we prove that the finite-index subgroup $G'$ of $G$ is
  not a congruence subgroup, i.e.\ does not contain
  $\mathrm{Stab}_G(n)$ for any $n\in \mathbb{N}$. From
  Proposition~\ref{pro:abelianisation}, it follows that
  $c^{-1} t_{n} \equiv c^{-1} b \not \equiv 1 \pmod{G'}$.  Hence $G'$
  does not contain $\mathrm{Stab}_G(n)$.
\end{proof}


\subsection{Weakly maximal subgroups}

Let $G \in \mathscr{C}$ be branch.  We recall that the \emph{parabolic
  subgroups} of $G$ are the stabilisers of the boundary points
$\ell \in \partial T$.  A subgroup of $G$ is \emph{weakly maximal} if
it is maximal among the subgroups of infinite index.  For a finitely
generated regular branch group, \cite[Theorem~1.1]{arxiv} shows that
any finite subgroup is contained in uncountably many weakly maximal
subgroups. It then follows that there are uncountably many
non-parabolic weakly maximal subgroups. This applies accordingly to
$G$, by using the finite subgroup $\langle a \rangle$ acting fix-point
freely on~$\partial T$. Note also that Corollary~\ref{uniqueaction}
allows us to consider all branch actions, thus
\cite[Theorem~1.3]{arxiv} is partly generalised:

\begin{corollary} \label{weakly} Let
  $G \in \mathscr{C}_\mathrm{reg}$.  Then there exist
  uncountably many $\mathrm{Aut}(G)$-orbits of weakly maximal
  subgroups of~$G$, all distinct from the orbits of parabolic
  subgroups under any branch action of $G$ on any spherically
  homogeneous rooted tree.
\end{corollary}


\section{Theta maps} \label{sec:theta}

\subsection{Length reduction} \label{sec:theta-def} Let
$G = \langle a,\mathbf{b}^{(1)},\ldots , \mathbf{b}^{(p)}\rangle \in
\mathscr{C}$
be in standard form.  We may assume that $r_1 \not = 0$ and that
$b_1 = b^{(1)}_1$ satisfies
\begin{equation} \label{equ:b1} \psi_1(b_1) =
  (a^{e_{1,1}},\ldots,a^{e_{1,p-1}},b_1) =
  (a,a^{e_{1,2}},\ldots,a^{e_{1,p-1}},b_1);
\end{equation}
see Lemma~\ref{dagger}.  We set
\begin{equation} \label{equ:def-n} n = \max \big\{ j \in
  \{1,\ldots,p-1\} \mid e_{1,j} \ne 0 \text{ in }
    \mathbb{Z}/p\mathbb{Z}  \big\}.
\end{equation}
Whereas we considered a slightly more general setting in \cite{Paper},
we suppose here from the outset that $G$ is a torsion group so that
$n\ge 2$.  This shortens some of the proofs.

In preparation for Section~\ref{sec:max-sub}, we recall
from~\cite[Section~4.2]{Paper} two length decreasing maps
$\Theta_1, \Theta_2 \colon G' \rightarrow G'$.  Clearly, $G'$ is a
subgroup of $\mathrm{Stab}_G(1)$.  Furthermore, every
$g\in \mathrm{Stab}_G(1)$ has a decomposition
\[
\psi_1(g)= (g_1,\ldots,g_p),
\]
where each $g_j \in U_{u_j}^G \cong G$ is an element of the upper
companion group acting on the subtree rooted at a first level vertex
$u_j$, $j \in \{1,\ldots,p\}$, and we define
\begin{equation} \label{eq:def-phi} \varphi_j\colon \mathrm{Stab}_G(1)
  \rightarrow \mathrm{Aut}(T_{u_j}) \cong \mathrm{Aut}(T), \quad
  \varphi_j(g)=g_j,
\end{equation}
using the natural identification of $T_{u_j}$ and~$T$.  
  In the proof of Theorem~\ref{theorem4.5} below we write
  $(g_1,\ldots,g_p)$ without warning in place of
  $g \in \mathrm{Stab}_G(1)$, as is customary to streamline certain
  computations.

We are interested in projecting, via $\varphi_p$, the first level
stabiliser $\mathrm{Stab}_M(1)$ of a subgroup $M \leq G$, containing
$b_1$ and an `approximation' $az \in a G'$ of $a$, to a subgroup of
$\mathrm{Aut}(T_{u_p})$.  Writing $\psi_1(z) =(z_1,\ldots,z_p)$, one
can show that
\[
\varphi_p({b_1}^{(az)^{-1}}) = a^{z_1^{-1}} = a[a,z_1^{-1}]
\]
and from this we define
\[
\Theta_1\colon G' \rightarrow G', \quad \Theta_1(z)=[a,z_1^{-1}].
\]

The map $\Theta_2$ is obtained similarly. As
  $e_{1,n} \ne 0$, we find $k \in \mathbb{Z}/p\mathbb{Z}$ such that
  $k e_{1,n} = 1$. One can show that
\[
\varphi_p \big( \big({b_1}^k \big)^{(az)^{p-n}} \big) = a^{z_{n+1}
  \cdots z_p} = a [a,z_{n+1} \cdots z_p]
\]
and we define
\[
\Theta_2\colon G'\rightarrow G', \quad
\Theta_2(z)=[a,z_{n+1} \ldots z_p].
\]

\begin{theorem} \label{theorem4.5} Let
  $G=\langle a,\mathbf{b}^{(1)}, \ldots, \mathbf{b}^{(p)} \rangle \in
  \mathscr{C}$
  be in standard form such that $r_1 \ne 0$ and \eqref{equ:b1} holds.
  Suppose further that $G$ is a torsion group.  Then the length
  $\partial(z)$ of an element $z\in G'$ decreases under repeated
  applications of a suitable combination of the maps $\Theta_1$ and
  $\Theta_2$ down to length $0$ or $2$.
\end{theorem}

\begin{proof}
  Let $z\in G'$.  We observe that $\partial(z)\ne 1$.  Suppose that
  $m = \partial(z) \ge 3$.  Then
  $z \in G' \subseteq \mathrm{Stab}_G(1)$ has a decomposition
  \[
  \psi_1(z) = (z_1,\ldots,z_p),
  \]
  and Lemma~\ref{lemma4.1} yields
  \[
  \partial(z_1)+ \partial(z_{n+1} \cdots z_p) \le m,
  \]
  where $n\ge 2$ is as defined in~\eqref{equ:def-n}.

  If $\partial(z_1)<\frac{m}{2}$ then $\partial(\Theta_1(z))<m$, and
  likewise if $\partial(z_{n+1} \cdots z_p)<\frac{m}{2}$ then
  $\partial(\Theta_2(z))<m$, and we apply induction.  Hence we may
  suppose that $m=2\mu$ is even and
  \begin{equation}\label{equ:mu-decomp}
    \partial(z_1) = \partial(z_{n+1} \cdots z_p) = \mu.
  \end{equation}

  We write $Z = z_{n+1} \cdots z_p = a^k (Z_1,\ldots,Z_p)$, for
  suitable $k\in \mathbb{Z}/p\mathbb{Z}$ and
  $(Z_1,\ldots,Z_p) \in \mathrm{Stab}_G(1)$ with
  $\sum_{i=1}^p \partial(Z_i) \leq \mu$; see Lemma~\ref{lemma4.1}.
  Consider
  \[
  \Theta_2(z)=[a,Z] = Z^{-a} Z = (Z_p^{\, -1} Z_1, Z_1^{\, -1} Z_2,
  \ldots, Z_{p-1}^{-1} Z_p).
  \]
  If $\partial(Z_1)\ge 1$, then $n \geq2$
  implies
  \[
  \partial\big((\Theta_2(z))_1\big)
  + \partial\big((\Theta_2(z))_{n+1}\cdots (\Theta_2(z))_p\big)
  = \partial(Z_p^{\, -1} Z_1) + \partial(Z_n^{\, -1} Z_p) <m.
  \]
  Consequently, $\partial((\Theta_2(z))_1)<\frac{m}{2}$ or
  $\partial((\Theta_2(z))_{n+1} \cdots (\Theta_2(z))_p)<\frac{m}{2}$
  and we are done by our earlier argument.  

  From now on suppose that~$Z_1 \in \langle a \rangle$.
  Applying~\eqref{equ:mu-decomp} to $\Theta_2(z)$ instead of $z$, we
  may assume that $\partial(Z_p) = \partial((\Theta_2(z))_1) = \mu$.  This
  implies that, in our usual notation,
  \[
  Z = a^k \, \big( (\mathbf{b}^{(k(1))})^{\boldsymbol{\beta}(1)} \big)^{a^*}
  \big( (\mathbf{b}^{(k(2))})^{\boldsymbol{\beta}(2)} \big)^{a^*}
  \cdots \big( (\mathbf{b}^{(k(\mu))})^{\boldsymbol{\beta}(\mu)}
  \big)^{a^*},
  \]
  where $k(i) \in \{1,\ldots, p\}$ for $1 \le i \le \mu$, the
  $\boldsymbol{\beta}(i) \in (\mathbb{Z}/p\mathbb{Z})^{r_{k(i)}}$ are
  suitable exponent vectors and the undeclared exponents $*$ of $a$
  are such that
  \begin{equation} \label{equ:z-p} Z_p = \big( \mathbf{b}^{(k(1))}
    \big)^{\boldsymbol{\beta}(1)} \big( \mathbf{b}^{(k(2))}
    \big)^{\boldsymbol{\beta}(2)} \cdots \big( \mathbf{b}^{(k(\mu))}
    \big)^{\boldsymbol{\beta}(\mu)}
  \end{equation}
  and $\{ Z_1, \ldots, Z_{p-1} \} \subseteq \langle a \rangle$;
  furthermore, $k(i) \ne k(i+1)$ for $1 \le i \le \mu-1$.  This
  implies $Z_n^{\, -1} Z_p = a^l Z_p$ for some $l\in \mathbb{Z}/p\mathbb{Z}$, hence
  \begin{equation*}
    \Theta_2(\Theta_2(z)) = [a,a^l Z_p]=[a,Z_p] = Z_p^{\, -a} Z_p.
  \end{equation*}

  We now repeat, for $Z_p = (Z_{p,1},\ldots,Z_{p,p})$, the argument
  applied earlier to~$Z$.  If $\partial(Z_{p,1}) \geq 1$, we are done
  by our earlier reasoning.  Otherwise we see that
  $\partial(Z_{p,p}) = \mu$, and \eqref{equ:z-p} implies
  $k(1)=k(2)=\ldots = k(\mu)=1$ leading to $\mu=1$, hence
  $m = 2\mu = 2$, a contradiction.
\end{proof}

We briefly comment that the above proof simplifies the corresponding
proof in \cite[Theorem~4.5]{Paper} for the multi-edge spinal groups
when $n\ge 2$.


\section{Maximal subgroups} \label{sec:max-sub} The cosets of
finite-index subgroups of a group $G$ form a base for the
\textit{profinite topology} on~$G$.  A subgroup $H$ of $G$ is
\textit{dense} with respect to the profinite topology if and only if
$G = NH$ for every finite-index normal subgroup $N$ of~$G$.  Thus
every maximal subgroup of infinite index in $G$ is dense and every
proper dense subgroup is contained in a maximal subgroup of infinite
index.

We consider a group $G \leq \mathrm{Aut}(T)$ acting on
  the regular $p$-adic rooted tree~$T$.  For any vertex $u$ of $T$,
  the group $U_u^G \le \mathrm{Aut}(T_u)$ maps isomorphically onto a
  group $G_u \leq \mathrm{Aut}(T)$ under the map induced by the
  natural identification of $T_u$ with~$T$.  Similarly, for a subgroup
  $M$ of $G$, we write $M_u$ for the corresponding image of $U_u^M$.
  If $G$ is fractal, then $G_u = G$ but nevertheless it is sometimes
  useful to write $G_u$ to emphasise that we are considering the
  isomorphic image of~$U_u^G$.  

  Taking into consideration that the notational convention
  in~\cite{Paper} is slightly different, we record some preliminary
  results.

\begin{proposition}[{\cite[Proposition~3.2]{Pervova4}}] \label{3.3.3}
  Let $T$ be a spherically homogeneous rooted tree and let $G \le
  \mathrm{Aut(T)}$ be a just infinite group acting transitively on each level
  of~$T$. Let $M$ be a dense subgroup of $G$ with respect to the
  profinite topology. Then
  \begin{enumerate}
  \item[(1)] the subgroup $M$ acts transitively on each level of the
    tree $T$,
  \item[(2)] for every vertex $u \in T$, the subgroup $M_u$ is dense
    in $G_u$ with respect to the profinite topology.
  \end{enumerate}
\end{proposition}

The following result is a direct generalisation of
  \cite[Proposition~5.2]{Paper}.  The proof of the latter, however,
  does not seem to contain all necessary details; these have
  now been worked out in a more general setting by Francoeur and
  Garrido following a strategy originally due to Pervova;
  see~\cite[Proposition~6.3]{AlejandraDominik}.

  \begin{proposition} \label{proper} Let $T$ be the regular $p$-adic
    rooted tree, and let $G \leq \mathrm{Aut}(T)$ be a branch group
    that is just infinite and fractal.  Let $M$ be a proper dense
    subgroup of $G$, with respect to the profinite topology. Then
    $M_u$ is a proper subgroup of $G_u = G$ for every
    vertex $u$ of~$T$.
\end{proposition}

We now proceed in similar fashion to~\cite{Paper}.

\begin{proposition} \label{proposition5.3} Let
  $G=\langle a,\mathbf{b}^{(1)}, \ldots, \mathbf{b}^{(p)} \rangle \in
  \mathscr{C}$
  be in standard form. Suppose that $G$ is a torsion
    group, and let $M$ be a dense subgroup of $G$,
  with respect to the profinite topology.

  Then for each $j\in \{1,\ldots, p\}$ and
  $i\in \{1,\ldots,r_j\}$ there is a vertex $u$ of $T$ and an element
  $g\in \mathrm{Stab}_G(u)$, acting on $T_u$ as $a^{\ell}$ for some
  $\ell \in \mathbb{Z}/p\mathbb{Z}$ under the natural identification
  of $T_u$ and $T$, such that
  \begin{enumerate}
  \item[(i)] $(M^g)_u=(M_u)^{a^{\ell}}$ is a dense subgroup of
    $G_u = G$, 
  \item[(ii)] there exists
    $c \in (M_u)^{a^{\ell}} \cap \langle
    b^{(j)}_1,\ldots,b^{(j)}_{r_j} \rangle$
    such that $\varepsilon_{b^{(j)}_i}(c) \ne 0$.
  \end{enumerate}
\end{proposition}

\begin{proof}
  By symmetry it suffices to prove the statement for $j=1$ and $i=1$,
  assuming $r_1 > 0$.  For notational simplicity, we write
  $b_1 = b^{(1)}_1, \ldots, b_{r_1} = b^{(1)}_{r_1}$.  It suffices to
  produce $u$ such that (ii) holds, as with $G$ being fractal, the existence
  of $g$ is automatic and Proposition~\ref{3.3.3} yields~(i).

  Since $\lvert G:G' \rvert$ is finite, $G'$ is open in the profinite
  topology. Thus we find $x\in M\cap b_1G'$.  In particular
  $x\in \mathrm{Stab}_G(1)$ with
  $\varepsilon_{b_1}(x) \ne 0$ in
    $\mathbb{Z}/p\mathbb{Z}$. We argue by induction on
  $\partial(x)\ge 1$.
 
  First suppose that $\partial(x)= 1$. Then $x$ has the form
  $x = c^{a^{\ell}}$, where
  $c \in \langle b_1,\ldots, b_{r_1} \rangle$ with
  $\varepsilon_{b_1}(c) \ne 0$.  Thus choosing the
  vertex $u$ to be the root of the tree~$T$, we have
  $c \in M^{a^{-\ell}} = (M_u)^{a^{-\ell}}$.
 
  Now suppose that $m = \partial(x) \ge 2$.  We first determine a
  suitable vertex $u_\omega$ at level $1$ or $2$, as follows.  Observe
  that 
  \[
  \varepsilon_{b_1}(\varphi_1(x))+ \ldots +
  \varepsilon_{b_1}(\varphi_p(x)) = \varepsilon_{b_1}(x) \ne 0.
  \]
  If there exists $i \in \{1,\ldots,p\}$ with
  $\varepsilon_{b_1}(\varphi_i(x)) \ne 0$ and
  $\partial(\varphi_i(x))<\partial(x)$, then we fix $\omega = i$.  If
  not, then for each $j \in \{1,\ldots,p\}$ with
  $\varepsilon_{b_1}(\varphi_j(x)) \ne 0$ we have
  $\partial(\varphi_j(x))=\partial(x)$.  But Lemma~\ref{lemma4.1}
  shows that $\sum_{i=1}^p \partial(\varphi_i(x))\le \partial(x)$;
  hence there exists a unique $j$ such that
  $\partial(\varphi_j(x)) =\partial(x)$ and
  $\varphi_i(x) \in \langle a \rangle$ for all remaining indices
  $i \ne j$.  This implies that $x$ is of the form
  \[
  x = \big( (\mathbf{b}^{(k(1))})^{\boldsymbol{\beta}(1)} \big)^{a^*}
  \cdots \big( (\mathbf{b}^{(k(m))})^{\boldsymbol{\beta}(m)}
  \big)^{a^*},
 \]
 where the exponents $*$ are such that
 $\varphi_j(x) = (\mathbf{b}^{(k(1))})^{\boldsymbol{\beta}(1)} \cdots
 (\mathbf{b}^{(k(m))})^{\boldsymbol{\beta}(m)} \in
 \mathrm{Stab}_G(1)$.
 Furthermore $\varepsilon_{b_1}(\varphi_j(x)) \ne 0$
 implies that $\varepsilon_{b_1}(\varphi_{jp}(x)) \ne 0$
 and $\partial(\varphi_{jp}(x))<\partial(x)$, where
 $\varphi_{jp}(x)=\varphi_p(\varphi_j(x))$.  We fix $\omega = jp$.

 We proceed in our analysis with 
 \[
 \tilde{x} = \varphi_\omega(x) \in M_{u_\omega} \leq G_{u_\omega}
 = G,
 \]
satisfying $\varepsilon_{b_1}(\tilde{x}) \ne 0$ and
$\partial(\tilde{x}) < \partial(x)$. First suppose that
$\tilde{x} \in \mathrm{Stab}_{G_{u(\omega)}}(1)$, where we write
$u(\omega)=u_\omega$ for readability.  By Proposition~\ref{proper},
the subgroup $M_{u_\omega}$ is dense in
$G_{u_\omega} = G$, and the result follows by induction.
 
 Now suppose that
 $\tilde{x} \not \in \mathrm{Stab}_{G_{u(\omega)}}(1)$.  For
 $l\in \{1,\ldots,p\}$ we claim  
 \begin{equation} \label{equ:p-th-power-congr}
   \varepsilon_{b_1}(\varphi_l(\tilde{x}^p)) =
   \varepsilon_{b_1}(\tilde{x}) \ne 0.
 \end{equation}
 To see this, observe that $\tilde{x}$ is of the form
 \begin{equation} \label{equ:tilde-x-1} \tilde{x}= a^k h = a^k
   (h_1,\ldots,h_p),
 \end{equation}
 where $k = \varepsilon_a(x) \ne 0$ and
 $h \in \mathrm{Stab}_{G_{u(\omega)}}(1)$ with
 $\psi_1(h) = (h_1,\ldots,h_p)$.  Raising $\tilde{x}$ to the prime
 power $p$, we obtain
 \[
  \tilde{x}^p = (a^k h)^p = h^{a^{(p-1)k}} \cdots h^{a^k} h,
 \]
 and thus, for $l\in \{1,\ldots, p\}$,
 \begin{equation} \label{equ:tilde-x-2} \varphi_l(\tilde{x}^p) \equiv
   h_1 h_2 \cdots h_p \pmod{G_{u_\omega}'}.
 \end{equation}
 In view of \eqref{equ:tilde-x-1} and \eqref{equ:tilde-x-2}, we conclude that
 \eqref{equ:p-th-power-congr} holds. 
 
 Furthermore we observe, from the above equations and from the proof of Lemma~\ref{lemma4.1}, that
 \[
  \partial(\varphi_l(\tilde{x}^p)) \le \partial(\tilde{x}) < \partial(x).
 \]
 If $\varphi_l(\tilde{x}^p) \in M_{u(\omega l)}$
 belongs to $\mathrm{Stab}_{G_{u(\omega l)}}(1)$, we are done as
 before by induction.  If not, we apply the operation
 $y\mapsto \varphi_l(y^p)$ repeatedly.  Since $M$ is a torsion group,
 $x\in \mathrm{Stab}_M(1)$ and $\tilde{x}$ have finite order.
 Clearly, if $y\in G$ has finite order then $\varphi_l(y^p)$ has order
 strictly smaller than $y$. Thus after finitely many iterations, we
 reach an element
 \[
  \tilde{\tilde{x}}=\varphi_l(\varphi_l(\ldots
  \varphi_l(\varphi_{\omega}(x)^p)^p\ldots )^p)\in M_{u(\omega l\ldots l)}
 \]
 which, in addition to the inherited properties
 $\varepsilon_{b_1}(\tilde{\tilde{x}}) \ne 0$ and
 $\partial(\tilde{\tilde{x}})<\partial(x)$, satisfies
 $\tilde{\tilde{x}}\in \mathrm{Stab}_{G_{u(\omega l\ldots l)}}(1)$.  As
 before, the proof concludes by induction.
\end{proof}

The next result follows as in~\cite[Proposition~5.4]{Paper}, however
we give a slightly conciser version of the proof here.

\begin{proposition} \label{proposition5.4} Let
  $G=\langle a,\mathbf{b}^{(1)}, \ldots, \mathbf{b}^{(p)} \rangle \in
  \mathscr{C}$
  be in standard form such that $r_1 \ne 0$ and \eqref{equ:b1}
  holds.  Suppose further that $G$ is torsion.  Let $M$ be a
  dense subgroup of $G$, with respect to the profinite
  topology, and suppose that $b^{(1)}_1\in M$.  Then there exists a
  vertex $u$ of $T$ and an element $g\in \mathrm{Stab}_G(u)$ acting on
  $T_u$ as $h \in \mathrm{Stab}_G(1)$ under the natural identification
  of $T_u$ and $T$, such that
  \begin{enumerate}
  \item[(i)] $(M^g)_u=(M_u)^h$ is a dense subgroup of
    $G_u = G$,
  \item[(ii)] $a,b^{(1)}_1\in (M_u)^h$.
  \end{enumerate}
\end{proposition}

\begin{proof}
  It suffices to establish the existence of $u$ and $h$ such that (ii)
  holds, because $G$ is fractal. Since $G'$ is open and $M$ is dense
  in $G$, there is $z \in G'$ such that $az \in M$.
  Recall that we denote the $p$th vertex at level $1$
    by~$u_p$.  The coordinate map $\varphi_p$ allows us to restrict
  $\mathrm{Stab}_M(1)$ to $M_{u_p}$, and $b^{(1)}_1\in M$ implies
  $b^{(1)}_1\in M_{u_p}$.

  Consider the theta maps $\Theta_1,\Theta_2$ defined in
  Section~\ref{sec:theta}, and $n\ge 2$ be as in~\eqref{equ:def-n}.
  By definition, $a\Theta_1(z)$ and $a\Theta_2(z)$ belong to
  $M_{u_p}$. Moreover, repeated application of $\varphi_p$ corresponds
  to repeated applications of $\Theta_1$ and $\Theta_2$. By
  Proposition~\ref{proper} and Theorem~\ref{theorem4.5}, we may assume
  that $\partial(z)\in \{0,2\}$.  

  If $\partial(z)=0$ we take $h=1$ and there is nothing further to
  prove.  Suppose now that $\partial(z)=2$, and write
  $\psi_1(z)=(z_1,\ldots,z_p)$.  We distinguish between two cases.

  \medskip

  \noindent \underbar{Case 1}: $\partial(z_1) \in \{0, 2\}$. Then
  $\partial(z_1) = 0$ or $\partial(z_{n+1} \cdots z_p)=0$ so that
  $\Theta_1(z) = 1$ or $\Theta_2(z) = 1$, and again there is nothing
  further to prove.

  \medskip

  \noindent \underbar{Case 2}: $\partial(z_1)=1$.  Then
  $\tilde{z} = \Theta_1(z) = z_1^{\, a} z_1^{\, -1}$ satisfies
  $\partial(\tilde{z}) = 2$.  If $\partial(\tilde{z}_1) = 0$ or
  $\partial( \tilde{z}_{n+1} \cdots \tilde{z}_p )=0$, we proceed as in
  Case~1.  Thus we may assume that
  $z_1^{-1} = a^{\varepsilon_a(z_1^{-1})}(a^*,\ldots, a^*,h^{-1})$ for
  suitable $h \in \langle \mathbf{b}^{(j)}\rangle$ with
  $j \in \{1, \ldots, p\}$.  If $j\ne 1$, then we deduce that
  $\hat{z} = \Theta_2(\tilde{z}) = h^a h^{-1}$ satisfies
  $\partial(\hat{z}_1) = 0$ or
  $\partial(\hat{z}_{n+1} \cdots \hat{z}_p ) = 0$, and we proceed as
  in Case~1.  Now suppose that $j = 1$.  Then
  $hah^{-1} = a \hat{z} = a\Theta_2(\Theta_1(z))$ and, observing that
  $b_1^{(1)}$ commutes with $h$, we conclude that
  $a, b_1^{(1)} \in (M_{u_{pp}})^h$.
\end{proof}

The proofs of the next two results follow the same logic as those of
\cite[Proposition~5.5 and Theorem~5.6]{Paper}, so we omit the proofs
here.

\begin{proposition} \label{isG} Let $G \in \mathscr{C}$ be a torsion
  group, and let $M$ be a dense subgroup of $G$, with respect to the
  profinite topology.  Then there exists a vertex $u$ of $T$ such that
  $M_u = G_u = G$.
\end{proposition}

\begin{theorem} \label{last} Let $G \in \mathscr{C}$ be a torsion
  group.  Then $G$ does not contain any proper dense subgroups, with
  respect to the profinite topology.  Equivalently, $G$ does not contain
  maximal subgroups of infinite index.
\end{theorem}

Recall that two groups $G$ and $H$ are (abstractly)
\emph{commensurable} if there exist finite-index subgroups $K \le G$
and $L \le H$ with $K \cong L$.

\begin{corollary} Let $H$ be a group that is commensurable with a
  torsion group $G \in \mathscr{C}$.  Then $H$ does not contain
  maximal subgroups  of infinite index.
\end{corollary}
 
\begin{proof}
  The proof is essentially the same as that of
  \cite[Corollary~1.3]{Paper}. However we do note here that
  $G=\langle a,\mathbf{b}^{(1)}, \ldots , \mathbf{b}^{(p)}\rangle$
  contains multi-edge spinal subgroups
  $G_j =\langle a,\mathbf{b}^{(j)} \rangle$ for $j\in \{1,\ldots,p\}$
  with $r_j\ne 0$.  Within these lie the associated GGS-groups
  $G_{j,i} = \langle a,b^{(j)}_i \rangle$ for $i\in \{1,\ldots, r_j\}$.
  These associated GGS-groups feature in the proof, as
  in~\cite{Paper}.
 \end{proof}

 We have established part (2) of Theorem~\ref{thm:main-result}.

 
\section{Irreducible representations} \label{sec:irr-reps}

In this section we prove Theorem~\ref{primitive} and
Theorem~\ref{thm:PassmanTemple}.  Throughout, let $F$ denote the prime
field $\mathbb{Z}/p\mathbb{Z}$.  Let $G \in \mathscr{C}$ be just
infinite, acting on the $p$-adic regular tree $T$.  We combine the
strategies laid out in~\cite{Bartholdi, Sidki1} to demonstrate when
the tree enveloping algebra $\mathfrak{A}_G$, a proper ring quotient
of $F[G]$, is primitive. The latter will imply that $G$ has faithful
irreducible representations over $F$.

\subsection{Preliminaries} We consider the $F$-vector space
$F\langle \!\langle\partial T \rangle \!\rangle$ on the
basis~$\partial T$, the boundary of~$T$. The action of $G$ on
$\partial T$ extends to an $F$-linear representation of the group
algebra
\[
\chi \colon F[G] \rightarrow \mathrm{End}(F \langle\!\langle \partial
T \rangle\!\rangle)
\]
which is injective on $G$. The \emph{tree enveloping algebra} of $G$
is the image $\mathfrak{A}_G$ of $F[G]$ under~$\chi$. It was
implicitly introduced by Sidki~\cite{Sidki1}, albeit in a different
form.

We collect some properties of the $F$-algebra $\mathfrak{A}_G$
from~\cite{Bartholdi}, with statements adapted to our setting. For
conciseness we include certain definitions and proofs, though altered
to suit our notation and purposes.  

Recall that an $F$-algebra $A$ is called \emph{just infinite}, if
$\dim_F A = \infty$ and every non-zero two-sided ideal has finite
codimension.  The Jacobson radical $\mathrm{Jac}(A)$ is the two-sided
ideal $\mathrm{Jac}(A) = \bigcap_X \mathrm{Ann}(X)$, where $X$ ranges
over all simple right $A$-modules.  For a one-sided ideal $I$ of~$A$,
the maximal two-sided ideal contained in $I$ is called the \emph{core}
of $I$, denoted by $\mathrm{core}(I)$.  The algebra $A$ is
\emph{primitive} if it has a faithful irreducible right module, or
equivalently a maximal right ideal with trivial core.  The algebra $A$
is \emph{semiprimitive} if its Jacobson radical is trivial.  Finally,
let $\mathrm{Aug}(\mathfrak{A}_G)$ denote the image of the
augmentation ideal of $F[G]$ in $\mathfrak{A}_G$.

\begin{lemma}[{\cite[Lemma~3.8,
      Theorem~3.9]{Bartholdi}}] \label{Lemma3.8} Let
  $G \in \mathscr{C}$ be just infinite.  Then its tree enveloping
  algebra $\mathfrak{A}_G$ is just infinite.
\end{lemma}

\begin{corollary}[{\cite[Lemma~3.15]{Bartholdi} and
    \cite[Corollary~4.4.3]{Sidki1}}] \label{C4.4.3} Let
  $G \in \mathscr{C}$ be just infinite. Then
  $\mathrm{Jac}(\mathfrak{A}_G)$ is either the zero ideal or equal to
  $\mathrm{Aug}(\mathfrak{A}_G)$.
\end{corollary}

\begin{proof}
  Suppose that $\mathrm{Jac}(\mathfrak{A}_G)\ne \{0\}$.  As the algebra
  $\mathfrak{A}_G$ is just infinite, it follows that
  $\dim_F(\mathfrak{A}_G/\mathrm{Jac}(\mathfrak{A}_G)) < \infty$.
  Since $\lvert F \rvert < \infty$, this implies
  $\lvert \mathfrak{A}_G / \mathrm{Jac}(\mathfrak{A}_G) \rvert <\infty$.

  Now $\chi \colon F[G] \rightarrow \mathfrak{A}_G$ can
    be factored as
    \[
    F[G] \rightarrow F[\widetilde{G}]\rightarrow \mathfrak{A}_G,
    \]
    where $\widetilde{G}$ is the closure of $G$ in the profinite group
    $\text{Aut}(T)$. Observe that $\widetilde{G}$ is a finitely
    generated pro-$p$ group. We obtain an induced group homomorphism
    \[
    \widetilde{\chi} \colon \widetilde{G} \rightarrow
    (\mathfrak{A}_G/\mathrm{Jac}(\mathfrak{A}_G))^*
    \]
    from $\widetilde{G}$ to the unit group of the finite algebra
    $\mathfrak{A}_G/\mathrm{Jac}(\mathfrak{A}_G)$.  Set
    $\widetilde{N} = \ker{\widetilde{\chi}}$, $N=G\cap \widetilde{N}$
    and $I = \chi^{-1}(\mathrm{Jac}(\mathfrak{A}_G))$.  Then
    $\widetilde{N}$ is normal and of finite index in~$\widetilde{G}$
    and thus $\widetilde{N}$ is open in $\widetilde{G}$; see
    \cite[Theorem 1.17]{DDMS99}. Hence $\widetilde{G}/\widetilde{N}$
    is a finite $p$-group. Consequently $G/N$ is a finite $p$-group.
  
    As $\langle x-1 \mid x\in N \rangle \subseteq I$, the homomorphism
    $ F[G] \rightarrow F[G]/I\cong
    \mathfrak{A}_G/\mathrm{Jac}(\mathfrak{A}_G)$
    factors through $F[G/N]$.  Since $G/N$ is a finite $p$-group, it
    follows that $\mathrm{Jac}(F[G/N]) = \mathrm{Aug}(F[G/N])$ and
    therefore
    $\mathrm{Jac}(\mathfrak{A}_G) = \mathrm{Aug}(\mathfrak{A}_G)$.
\end{proof}

\begin{proposition}[{\cite[Proposition~4.22]{Bartholdi}}]
  Let $G \in \mathscr{C}$ be just infinite.  If $\mathfrak{A}_G$ is
  semiprimitive, then it is primitive.
\end{proposition}

\begin{proof}
  Let $\mathcal{M}$ denote the collection of all maximal right ideals
  of $\mathfrak{A}_G$.  Suppose that $\mathfrak{A}_G$ is
  semiprimitive, i.e., that
  $\bigcap_{M \in \mathcal{M}} M = \mathrm{Jac}(\mathfrak{A}_G) = 0$.
  We need to produce an $M \in \mathcal{M}$ with
  $\mathrm{core}(M) = 0$.  For this it suffices to show that, if
  $M \in \mathcal{M}$ with $\mathrm{core}(M) \ne 0$, then
  $\mathrm{core}(M) = \mathrm{Aug}(\mathfrak{A}_G)$.

  Let $M \in \mathcal{M}$ with $I = \mathrm{core}(M) \ne 0$.  Since
  $\mathfrak{A}_G$ is just infinite and $\lvert F \rvert < \infty$, it
  follows that $\mathfrak{A}_G/I$ is finite.  As in the previous
  proof, there is a normal subgroup $N$ of $G$ such that the
  epimorphism $F[G] \rightarrow \mathfrak{A}_G/I$ factors through the
  group algebra $F[G/N]$ of the finite $p$-group~$G/N$.

  Write $\overline{J}$ for the image of
    $J = \chi^{-1}(I) \subseteq F[G]$ in $F[G/N]$ and observe that
    $\overline{J}$ is a maximal two-sided ideal of $F[G/N]$. However
    in the finite local ring $F[G/N]$ the augmentation ideal
    $\mathrm{Aug}(F[G/N])$ is the only maximal two-sided ideal.  Hence
    we obtain $\overline{J} = \mathrm{Aug}(F[G/N])$ and therefore
    $J= \mathrm{Aug}(F[G])$. Hence
    $I =\chi(J)=\mathrm{Aug}(\mathfrak{A}_G)$.
\end{proof}


\subsection{The depth function}

Let
$G = \langle a, \mathbf{b}^{(1)},\ldots , \mathbf{b}^{(p)}\rangle \in
\mathscr{C}$
be in standard form, acting on the regular $p$-adic rooted tree $T$
with vertices labelled by elements of~$\overline{X}$.
  Recall that $\overline{X}$ consists of all words in the alphabet
  $X = \{1,\ldots,p\}$, and the length of a word
  $\omega \in \overline{X}$ is denoted by $\lvert \omega \rvert$.
Recall that every $g \in G$ can be expressed as
$(g_1,\ldots , g_p) a^{\varepsilon_a(g)}$, where
$\varepsilon_a(g) \in \mathbb{Z}/p\mathbb{Z}$ and
$(g_1,\ldots , g_p) \in G^p$ is short for
$\psi_1^{-1}(g_1,\ldots , g_p)$.  Of course, the decomposition can be
reiterated, giving
$g_\omega = (g_{\omega1},\ldots, g_{\omega
    p})a^{\varepsilon_a(g_\omega)}$
  for any word $\omega \in \overline{X}$.

As in~\cite{Sidki1}, we define a \emph{depth function} 
\begin{multline*}
  d \colon G \rightarrow \mathbb{N}_0, \quad d(g) = \min \{ d
  \in \mathbb{N}_0 \mid \forall \omega \in \overline{X}
  \text{ with } \lvert \omega \rvert \ge d : \\
  (g_{\omega 1},\ldots,g_{\omega p}) \in \langle \mathbf{b}^{(1)} \rangle \cup
  \ldots \cup \langle \mathbf{b}^{(p)} \rangle\}.
\end{multline*} We remark that the function is well-defined, because
by Lemma~\ref{lemma4.1} for any given $g \in G$ the lengths
$\partial(g_\omega)$ decrease down to $0$ or $1$ as
$\lvert \omega \rvert \to \infty$.  Furthermore, we
observe for every $g \in G$:
\begin{enumerate}
\item[$\circ$] $d(g) = d((g_1,\ldots, g_p))$,
\item[$\circ$] $d(g)=0$ if and only if
  $g \in \langle \mathbf{b}^{(1)}\rangle \langle a \rangle \cup \ldots
  \cup \langle \mathbf{b}^{(p)}\rangle \langle a\rangle$,
\item[$\circ$] if $d(g) \ne 0$, then
  $d(g)= \max\{d(g_i) \mid 1 \le i \le p \} + 1$.
\end{enumerate}

We notice that the natural embedding of groups
\[ 
G \hookrightarrow (G \times \overset{p}{\ldots} \times G) \rtimes
\langle a \rangle, \quad g \mapsto \big( (g_1,\ldots,g_p),
a^{\varepsilon_a(g)} \big)
\]
induces a natural embedding of $F$-algebras
\[
\mathfrak{A}_G \hookrightarrow (\mathfrak{A}_G\times \overset{p}{\ldots} \times
\mathfrak{A}_G) \rtimes F\langle a\rangle,
\]
where each $v$ of $\mathfrak{A}_G$ is mapped to
$v_0 + v_1a + \ldots + v_{p-1} a^{p-1}$ with
$v_i=(v_{i,1},\ldots,v_{i,p})\in \mathfrak{A}_G\times
\overset{p}{\ldots}\times \mathfrak{A}_G$ for $0\le i\le p-1$.  

Let
$(\mathfrak{A}_G\times \overset{p}{\ldots} \times \mathfrak{A}_G)^{\circ}$
denote the image of
$F\langle \mathbf{b}^{(1)}\rangle \cup \ldots \cup F\langle
\mathbf{b}^{(p)} \rangle \subseteq F[G]$
in $\mathfrak{A}_G\times \overset{p}{\ldots} \times \mathfrak{A}_G$. Though as the map $\chi$ from Section 6.1 is injective on $G$, we will often identify elements of $G$ with their images in $\mathfrak{A}_G$.

The depth function $d \colon G \rightarrow \mathbb{N}_0$ now extends to
$\mathfrak{A}_G$ as follows: for
$v = v_0 + v_1a + \ldots + v_{p-1}a^{p-1} \in \mathfrak{A}_G$, where
$v_i = (v_{i,1},\ldots, v_{i,p}) \in \mathfrak{A}_G\times
\overset{p}{\ldots} \times \mathfrak{A}_G$
for $1 \le i \le p-1$, we define recursively
\begin{multline*}
  d(v)= \max \big( \{0 \} \cup \{ d(v_{i,j}) + 1 \mid 0 \le i \le p-1
  \text{ and } 1 \le j \le p \\
  \text{ such that } v_i\not \in (\mathfrak{A}_G\times
  \overset{p}{\ldots} \times \mathfrak{A}_G)^{\circ} \} \big).
\end{multline*}
We observe that if $v=v_0$ has $d(v)\ne 0$, then $d(v)>d(v_{0j})$
for $1\le j\le p$.


\subsection{Invertibility} 

With reference to Remark \ref{genGS}, let $H = \langle a,b \rangle$ be a generalised Gupta--Sidki group
acting on the regular $p$-adic rooted $T$; recall that $a$ denotes the
rooted automorphism of order $p$ and $b$ a directed automorphism
defined recursively by
\begin{equation} \label{equ:psi1-b-decomp} \psi_1(b) =
  (a^{e_1},\ldots,a^{e_{p-1}},b), \, \text{where
    $\{e_1,\ldots,e_{p-1}\} = \{1,\ldots,p-1\}$.}
\end{equation}

Let
$G = \langle a, \mathbf{b}^{(1)}, \ldots, \mathbf{b}^{(p)}
\rangle \in \mathscr{C}$
be in standard form and just infinite.  Suppose that $H \le G$, and
say, without loss of generality, that $b = b^{(1)}_1$.  We denote by
$\mathfrak{T}$ the tree enveloping algebra of the Sylow-pro-$p$
subgroup $S \le \mathrm{Aut}(T)$ described in \eqref{equ:Sylow-pro-p} and
write $\mathfrak{A}_G \subseteq \mathfrak{T}$ for the tree
enveloping algebra of $G$, as usual.

For $v \in \mathfrak{T}$, we set
$v^{[1]} = (v,\overset{p}{\ldots},v) \in \mathfrak{T}\times
\overset{p}{\ldots}\times \mathfrak{T} \subseteq \mathfrak{T}$
and recursively, for $i \ge 2$,
\[
v^{[i]}=(v^{[i-1]},\overset{p}{\ldots},v^{[i-1]}) \in
\mathfrak{T}\times \overset{p^i}{\ldots}\times \mathfrak{T} \subseteq
\mathfrak{T}.
\]
Further, identifying elements of $G$ with their images in $\mathfrak{A}_G$, we introduce the notation
\begin{equation} \label{equ:a-b-star}
  \begin{split}
    a_* & = (a-1)^{p-1}=1+a+\ldots +a^{p-1} \in \mathfrak{A}_G, \\
    b_* & = (b - 1)^{p-1}=1+ b + \ldots + b^{p-1} \in
    \mathfrak{A}_G.
\end{split}
\end{equation}
Observe that for $(v_1,\ldots,v_p)\in \mathfrak{T}\times
\overset{p}{\ldots}\times \mathfrak{T} \subseteq \mathfrak{T}$,
\[
a_*(v_1,\ldots,v_p)a_*=(v_1+ \ldots +v_p)^{[1]}a_*. 
\]
 



In the following proof, the idea of the first half comes from
Sidki~\cite[Proposition~5.2]{Sidki1}, and was already used by Vieira~\cite[Theorem~1]{Vieira} for the special case of \emph{the} generalised Gupta--Sidki group.

\begin{lemma} \label{GS} Let
  $G= \langle a, \mathbf{b}^{(1)},\ldots ,
  \mathbf{b}^{(p)}\rangle \in \mathscr{C}$
  be just infinite, and suppose that $\langle a,b \rangle$, for
  $b=b^{(1)}_1$ as in \eqref{equ:psi1-b-decomp}, is a generalised
  Gupta--Sidki group.  Then $1+ba_*$ is not invertible in the tree
  enveloping algebra~$\mathfrak{A}_G$.
\end{lemma}

\begin{proof}
  Let $\eta=1+ba_*$. We suppose that $\eta$ is invertible in the tree enveloping algebra
  $\mathfrak{T}$ of the Sylow-pro-$p$
subgroup $S \le \mathrm{Aut}(T)$.  Then, as the $a \mapsto a$ and
  $b \mapsto b^{-1}$ induce an automorphism of~$G$, we see that
  $\mu= 1+b^{-1}a_*$ is also invertible in~$\mathfrak{T}$.
  By~\cite[Lemma~2]{Vieira}, we may express the inverses as
  \[
  \eta^{-1} = 1-\rho a_* \qquad \text{and} \qquad \mu^{-1}=1-\sigma a_*,
  \]
  where
  $\rho, \sigma \in \mathfrak{T} \times \overset{p}{\ldots} \times
  \mathfrak{T} \subseteq \mathfrak{T}$ such that
  \begin{equation} \label{eq:1}
    \rho(b+a_*)^{[1]}  = b  \qquad \text{and} \qquad
    \sigma(b^{-1}+a_*)^{[1]} = b^{-1}. 
  \end{equation}
  
  Multiplying on the right by
  $(\mu^{-1})^{[1]} = (1-\sigma a_*)^{[1]}$, we obtain
  from the first equation in~\eqref{eq:1},
  \[
  \rho b^{[1]} = \rho (b \mu)^{[1]} (\mu^{-1})^{[1]} = b(1-\sigma a_*)^{[1]}
  \]
  and hence
  \begin{equation} \label{eq:3} \rho=b(1-\sigma
    a_*)^{[1]}(b^{-1})^{[1]}.
  \end{equation}
  Similarly, we deduce
  \begin{equation} \label{eq:4} \sigma = b^{-1}(1-\rho a_*)^{[1]}b^{[1]}.
  \end{equation}

  Substituting (\ref{eq:4}) in (\ref{eq:3}) gives
  \begin{equation} \label{equ:rho-equation}
    \rho=b(b^{-1})^{[1]}-b(b^{-1})^{[1]} (1 - \rho
      a_*)^{[2]} b^{[2]}a_*^{[1]}(b^{-1})^{[1]}.
  \end{equation}

  Assume for a contradiction that
  $\eta^{-1} \in \mathfrak{A}_G$, and hence
  $\rho=(\rho_1,\ldots, \rho_p)\in \mathfrak{A}_G\times
  \overset{p}{\ldots}\times \mathfrak{A}_G$.
  Recalling the depth function $d$, we observe that
  $d(\rho) \ge d(\rho_p)$, where $\rho_p$ can be expressed according
  to~\eqref{equ:rho-equation} as
  \begin{align*}
    \rho_p & = 1-(1-\rho a_*)^{[1]}b^{[1]}a_*b^{-1} \\
           & = 1-(1-\rho
             a_*)^{[1]}b^{[1]}(b^{-1}+b^{-a^{-1}}a+\ldots +b^{-a}a^{p-1}).
  \end{align*}
  Now writing $\rho_p=v_0+v_1a+\ldots +v_{p-1}a^{p-1}$ with
  $v_i\in \mathfrak{A}_G\times \overset{p}{\ldots}\times
  \mathfrak{A}_G$
  for $0\le i\le p-1$, we obtain $d(\rho_p) \ge d(v_0)$, where
  \[
  v_0 = 1-(1-\rho a_*)^{[1]}b^{[1]}b^{-1}.
  \]
  Writing $v_0 = (v_{0,1},\ldots,v_{0,p})$, we see that
  $v_{0,p} = \rho a_*$ and hence $d(v_0)\ge d(v_{0,p})=d(\rho)$.  The
  inequalities $d(\rho)\ge d(\rho_p)\ge d(v_0)\ge d(\rho)$ imply
  equality throughout and hence $d(\rho)=0$ which is equivalent
  to
  $\rho \in F \langle \mathbf{b}^{(j)} \rangle \subseteq
  \mathfrak{A}_G$ for some $j\in \{1,\ldots, p\}$. Recall that as before we identify $F \langle \mathbf{b}^{(j)} \rangle$ with its image in $\mathfrak{A}_G$.

  If $j=1$, then $\rho=\rho_p=v_0=v_{0,p}=\rho a_*$, and therefore
  $\rho=0$, giving us the required contradiction.  Now suppose that
  $j\ne 1$.  According to~\eqref{equ:rho-equation}, the coordinates of
  $\rho = (\rho_1, \ldots, \rho_p)$ are
  \begin{align*}
    \rho_i & = a^{e_i}b^{-1} - a^{e_i}b^{-1} (1-\rho
             a_*)^{[1]} b^{[1]} a_* b^{-1} \quad \text{for $1 \leq i \leq
             p-1$,} \\
    \rho_p & =1 - (1-\rho a_*)^{[1]}b^{[1]}a_*b^{-1}.
  \end{align*} 
  Moreover, $\rho \in F \langle \mathbf{b}^{(j)}\rangle$ implies that
  $\rho _{p-j+1} \in F \langle \mathbf{b}^{(j)}\rangle$ and $\rho_1$,
  \ldots, $\rho_{p-j}$,$\rho_{p-j+2}$, \ldots,
  $\rho_p \in F \langle a\rangle $.  As $\rho_1,\ldots, \rho_{p-1}$
  all differ by a left multiple of $a$, it follows
  that
  $\rho_{p-j+1}\in F \langle \mathbf{b}^{(j)} \rangle \cap F \langle
  a\rangle =F$.  Consequently, there is a $\lambda \in F$ such that
  \begin{equation} \label{equ:p-j+1-formula}
    a^{e_{p-j+1}}b^{-1}(1-(\eta^{-1}b)^{[1]}a_*b^{-1})=\lambda.
  \end{equation}

  Our next aim is to express~\eqref{equ:p-j+1-formula} in matrix form
  via the embedding
  \begin{equation} \label{equ:algebra-embedding} \varphi \colon
    \mathfrak{A}_G\rightarrow \mathrm{Mat}_p(\mathfrak{A}_G)
  \end{equation}
  induced by
  \[
  a \mapsto
  \begin{pmatrix}
    0 & 1 & 0 &  \ldots & 0 \\
    0 & 0 & 1 &  \ldots & 0 \\
    \vdots & & \ddots & \ddots   & \vdots \\
    0 & \ldots & \ldots &  0 & 1 \\
    1 & 0 & \ldots & \ldots & 0
  \end{pmatrix}
  \quad \text{ and } \quad (c_1,\ldots,c_p) \mapsto
  \begin{pmatrix}
    c_1 &  & &   &  \\
    & c_2 &  &   &  \\
    & & \ddots &    &  \\
    &  &  &  c_{p-1} &  \\
    & & & \ & c_p
  \end{pmatrix};
  \]
  cf.\ \cite[Section 3]{Bartholdi}.  First we compute the relevant
  terms individually: 

  {\small
    \begin{align*}
      \varphi(a^{p-j+1}) & = 
                           \begin{pmatrix}
                             0 &  \text{Id}_{j-1}  \\
                             \text{Id}_{p-j+1} & 0
                           \end{pmatrix},
                                                 \quad \quad                                          
                                                 \varphi(a_*) = 
                                                 \begin{pmatrix}
                                                   1 & & \ldots & & 1  \\
                                                   & & & & \\
                                                   \vdots & & \ddots & & \vdots  \\
                                                   & & & & \\
                                                   1 & & \ldots & & 1
                                                 \end{pmatrix},
      \\ 
      \varphi(b^{-1}) & = 
                        \begin{pmatrix}
                          a^{-e_1} &  & &   &  \\
                          & a^{-e_2} &  &   &  \\
                          & & \ddots &    &  \\
                          &  &  &  a^{-e_{p-1}} &  \\
                          & & & \ & b^{-1}
                        \end{pmatrix},
      \\
      \varphi((\eta^{-1}b)^{[1]}) & = 
                                    \begin{pmatrix}
                                      \eta^{-1}b &  & &   &  \\
                                      & \eta^{-1}b &  &   &  \\
                                      & & \ddots &    &  \\
                                      &  &  &  \eta^{-1}b &  \\
                                      & & & \ & \eta^{-1}b
                                    \end{pmatrix}.
    \end{align*}
  }Combining these terms, we deduce from~\eqref{equ:p-j+1-formula},
  \begin{multline*}
    \lambda \text{Id}_p = \varphi
    \big(a^{e_{p-j+1}}b^{-1}(1-(\eta^{-1}b)^{[1]}a_*b^{-1}) \big) =
    \\
    \begin{pmatrix}
      -a^{-e_{p-j+2}}\eta^{-1}ba^{-e_1} &   \ldots &  -a^{-e_{p-j+2}}\eta^{-1}ba^{-e_{p-1}} & -a^{-e_{p-j+2}}\eta^{-1} \\
      \vdots &    & \vdots  & \vdots \\
      -b^{-1}\eta^{-1}ba^{-e_1} &  \ldots & -b^{-1}\eta^{-1}ba^{-e_{p-1}} & b^{-1}(1-\eta^{-1}) \\
      a^{-e_1}(1-\eta^{-1}ba^{-e_1})  & \ldots & -a^{-e_1}\eta^{-1}ba^{-e_{p-1}} & -a^{-e_1}\eta^{-1} \\
      \vdots &  & \vdots  & \vdots \\
      -a^{-e_{p-j+1}}\eta^{-1}ba^{-e_1} & \ldots &
      -a^{-e_{p-j+1}}\eta^{-1}ba^{-e_{p-1}} & -a^{-e_{p-j+1}}\eta^{-1}
    \end{pmatrix}.
  \end{multline*}
  Comparing the first and last terms of the first row, we
  see that $\lambda = 0$. Hence $\rho =0$, and therefore $\eta ^{-1}=1$, which cannot be possible. Thus, we have our contradiction.
\end{proof}

\begin{remark}
  In~\cite[Theorem~12.3]{Trends} Passman claims that
  Sidki's proof~\cite{Sidki1} for the Gupta--Sidki $3$-group follows
  through for Gupta--Sidki $p$-groups for all primes $p \geq 3$.
  However, \eqref{eq:1} does not hold for Gupta--Sidki
  $p$-groups, where $p\ge 5$.  One instead obtains the following equation
  \[
   \rho(b+a+a^{-1}+p-1)^{[1]}=b
  \]
  which prevents us from going on to the next deduction, as $a+a^{-1}+p-1\ne a_*$. We do not know of a way to skirt around
  this.
\end{remark}

\begin{proposition} \label{afterLemma6.4} Let
  $G=\langle a, \mathbf{b}^{(1)},\ldots , \mathbf{b}^{(p)}\rangle \in
  \mathscr{C}$
  be just infinite, and suppose that $\langle a,b \rangle$, for
  $b=b^{(1)}_1$ as in \eqref{equ:psi1-b-decomp}, is a generalised
  Gupta--Sidki group. Then $\mathrm{Jac}(\mathfrak{A}_G)=0$.
\end{proposition}

\begin{proof}
  As $\mathrm{Jac}(\mathfrak{A}_G)$ is a quasi-invertible ideal and as
  $a_*\in \mathrm{Aug}(\mathfrak{A}_G)$, it suffices by
  Corollary~\ref{C4.4.3}, to prove that $1+\lambda a_*$ is not
  invertible in $\mathfrak{A}_G$ for some
  $\lambda \in \mathfrak{A}_G$. The statement of the proposition now
  follows from Lemma~\ref{GS}.
\end{proof}


\subsection{Graded ideals}

Let
$G=\langle a, \mathbf{b}^{(1)},\ldots , \mathbf{b}^{(p)}\rangle \in
\mathscr{C}$
be in standard form and just infinite.  By conjugation, we may assume
that $r_1\ne 0$ and we write $b = b^{(1)}_1$.  We observe that the
elements $a_*, b_* \in \mathfrak{A}_G$ defined in \eqref{equ:a-b-star}
satisfy $a_*^2=b_*^2=0$.  Furthermore we write
$b_*=(a_*\epsilon_1,\ldots,a_*\epsilon_{p-1},b_*)$ where for
$1\le i\le p-1$,
\[
\epsilon_i= \begin{cases} 
0 & \text{if } e_i = 0 \text{ in } \mathbb{Z}/p\mathbb{Z}, \\
1 & \text{otherwise}.
\end{cases}
\]

Several preliminary results will be needed; the first is by
Smoktunowicz.

\begin{theorem}[{\cite[Lemma~4.24]{Bartholdi} and
      \cite[Theorem~1.1]{Agata}}] \label{Agata} Let
  $\mathfrak{I}=\oplus_{i=1}^{\infty}\mathfrak{I}_i$ be a graded
  algebra (without unit) generated in degree $1$.  Then the following
  are equivalent:
  \begin{enumerate}
  \item[(1)] $\mathrm{Jac}(\mathfrak{I})=\mathfrak{I}$; 
  \item[(2)] for every choice of $n,m\in \mathbb{N}$,
    all $n\times n$ matrices with entries in $\mathfrak{I}_m$ are
    nilpotent.
  \end{enumerate}
\end{theorem}

In the following, the proof is a slight modification of that
in~\cite[Proposition~4.1.1]{Sidki1}, due to the allowance of arbitrary
exponents $e_i$.

\begin{proposition} \label{poly} Let
    $G=\langle a, \mathbf{b}^{(1)},\ldots , \mathbf{b}^{(p)}\rangle
    \in \mathscr{C}$
    be just infinite. The subalgebra $F[b_*a_*]$ of $\mathfrak{A}_G$
  generated by $b_* a_*$ is a polynomial algebra. In
    particular, $b_*a_*$ is not nilpotent, and
    $\mathrm{Aug}(\mathfrak{A}_G)$ is not nil.
\end{proposition}

\begin{proof}
  We denote $b_* a_*$ by $X$ and $a_*b_*$ by $Y$.  Let
  \[
  N = \lvert \{e_i\mid e_i \ne 0 \text{ for } 1 \le i \le
  p-1\} \rvert,
  \]
  regarded as an element of~$(\mathbb{Z}/p\mathbb{Z})^*$.  By
  induction,
  \begin{align}
 \label{equ:dagger}
 \begin{split}
   X^{2j-1} & = N^{j-1} (\epsilon_1 Y^{j-1}a_*,
   \ldots,\epsilon_{p-1}Y^{j-1}a_*, X^{j-1}b_*)a_*
   \\
   X^{2j} & = N^{j-1} (\epsilon_1 Y^j, \ldots, \epsilon_{p-1}Y^j, NX^j) a_*
  \end{split}
 \end{align}
 for all $j\ge 1$.
 
 For a contradiction, assume that $X$ is a root of some non-zero
 polynomial $f \in F[t]$ of minimal degree
 $m = \deg(f) \in \mathbb{N}$, say. Note that $f$ has zero constant term, as otherwise it would imply that $X=b_*a_*$ were invertible. Now
 \[ 
 0=f(X)=(l_1,\ldots,l_{p-1},q)a_*
 \]
 for some $l_1,\ldots,l_{p-1},q \in \mathfrak{A}_G$, and from
 \eqref{equ:dagger} we deduce that
 \[
 0 = q = q_0(X) + q_1(X)b_*,
 \]
 where $q_0 ,q_1 \in F[t]$  satisfy $\max\{\deg(q_0),
 \deg(q_1) \} = \lfloor m/2 \rfloor$, and the polynomial $q_0$ also has zero constant term.
   
 Now since $0= qX = q_0(X)X$ and $0=qa_*=q_1(X)X$, it follows that $X$
 is a root of each of the polynomials $q_0(t)t, q_1(t)t$.  We conclude
 that $m \in \{1,2\}$ by minimality. Thus $f$ is either
 $t$ or $t^2+ct$ for some constant $c\in
   F$.
   It follows by a direct computation, with the aid of
   \eqref{equ:dagger}, that all these cases lead to a contradiction.
\end{proof}





\begin{proposition} \label{gradedzero} Let
  $G=\langle a, \mathbf{b}^{(1)},\ldots , \mathbf{b}^{(p)}\rangle \in
  \mathscr{C}$
  be in standard form and just infinite.  If
  $\mathrm{Aug}(\mathfrak{A}_G)$ is a graded algebra with
    the elements $a-1$ and $b_i^{(j)}-1$ for $1\le j\le p$,
    $\, 1\le i\le r_j$ being homogeneous, then
  $\mathrm{Jac}(\mathfrak{A}_G)=0$.
\end{proposition}

\begin{proof}
  It follows that $X=b_*a_*$ is homogeneous of degree $m$, say.
  Proposition~\ref{poly} shows that condition (2) of
  Theorem~\ref{Agata} does not hold for $n=1$. Therefore we obtain
  $\mathrm{Jac}(\mathrm{Aug}(\mathfrak{A}_G)) <
  \mathrm{Aug}(\mathfrak{A}_G)$.
  
  Now it follows from Corollary~\ref{C4.4.3} that
  $\mathrm{Jac}(\mathfrak{A}_G) = \mathrm{Aug}(\mathfrak{A}_G)$ or
  $\mathrm{Jac}(\mathfrak{A}_G) = 0$. However as
  $\mathrm{Jac}(\mathrm{Jac}(\mathfrak{A}_G))=\mathrm{Jac}(\mathfrak{A}_G)$
  (see~\cite[Exercise 4.7]{Lam}), we must have
  $\mathrm{Jac}(\mathfrak{A}_G)=0$.
\end{proof}

We can now give a proof of our main result.

\begin{proof}[Proof of Theorem~\ref{primitive}]
  From Proposition~\ref{afterLemma6.4} and
  Proposition~\ref{gradedzero}, it follows that
  $\mathrm{Jac}(\mathfrak{A}_G) = 0$ in both cases.  From the proof of
  Corollary~\ref{C4.4.3}, we see that this ensures the existence of a
  maximal right ideal $M$, with null core, which gives us a faithful
  irreducible representation of $\mathfrak{A}_G$ on the $F$-space
  $\mathfrak{A}_G/M$.
\end{proof}

\subsection{Irreducible representations}

We prove that when a branch generalised multi-edge
spinal group has a non-trivial irreducible representation, then it has
infinitely many.  The proof is similar to but somewhat conciser than
that of~\cite[Theorem 2.2]{PassmanTemple}.

\begin{theorem}
  Let $G \in \mathscr{C}_\mathrm{reg}$ and let
  $F$ be an algebraically closed field of characteristic~$p$.  If the
  group algebra $F[G]$ has at least one non-trivial irreducible module
  $M$, then $F[G]$ has
  infinitely many such irreducible modules.
\end{theorem}

\begin{proof}
  For notational convenience, set $H =\gamma_3(G)$. Furthermore set
  $L =\gamma_3(\mathrm{Stab}_G(1))$, the normal finite-index subgroup of
  $H$ with $L\cong H\times \overset{p} {\ldots} \times H$, by
  Proposition~\ref{sym}.  If $V_1,\ldots, V_p$ are irreducible
  $F[H]$-modules, then $V_1 \otimes \ldots \otimes V_p$ is an
  irreducible $F[L]$-module; see~\cite[Proposition
  8.4.2]{Robinson}. Let $F_0$ denote the trivial
  $1$-dimensional module, for both $F[G]$ and $F[H]$.
 
  We consider $F[H]$.  Let $U$ be the given non-trivial irreducible
  $F[G]$-module. Its restriction $U_H= U_1' \oplus \ldots \oplus U_t'$
  is a finite direct sum of irreducible $F[H]$-modules, by Clifford
  theory.  If all $U_i'\cong F_0$, then $U$ is an irreducible
  $F[G/H]$-module, as $H$ acts trivially on $U$.  But $G/H$ is a
  finite $p$-group and as the characteristic of $F$ is $p$, this
  implies that $U=F_0$, a contradiction. Therefore $F[H]$ has a
  non-trivial irreducible module~$V$, and by~\cite[Lemma
  2.1]{PassmanTemple} it suffices to show that $H$ has infinitely many
  such.
 
  By the branching property of $H$, we have, for all
    $m\in \mathbb{N}$, a subgroup $L_m \le \text{Stab}_H(m)$ which is
    isomorphic to $H\times \overset{p^m}{\ldots} \times H$ under
    $\psi_m$. We write
    $L_m = L_{m,1}\times \ldots \times L_{m,p^m}$, where
    $L_{m,i} \cong H$ for $1\le i\le p^m$.  The group $G$ permutes
  these factors by conjugation.  By viewing $V$ as a module for each
  $F[L_{m,i}]$, we can consider the $p^m+1$ irreducible
  $F[L_m]$-modules
 \[
 Y_j=V\otimes \overset{j}{\ldots}\otimes V \otimes F_0 \otimes
 \overset{p^m-j}{\ldots}\otimes F_0,
 \]
 where $j \in \{0,\ldots, p^m\}$.  Since $V\not \cong F_0$ and $G$ permutes the
 factors $L_{m,i}$ of $L_m$, it follows that $Y_0,\ldots, Y_{p^m}$ are
 in distinct orbits under the action of $G$.

 For each $j$, let $Z_j$ be an irreducible constituent of
 $Y_j^H$. Since the $Y_j$ are in distinct $H$-orbits, it follows that
 $Z_0,\ldots, Z_{p^m}$ are distinct irreducible $F[H]$-modules.  As
 this is true for all $m\in \mathbb{N}$, the result follows with the
 aid of~\cite[Lemma~2.1]{PassmanTemple}.
\end{proof}


\subsection*{Acknowledgement} The second
  author was funded by the Alexander von Humboldt Foundation. We thank Amaia Zugadi-Reizabal for
helpful remarks, e.g., concerning the generalised Gupta--Sidki groups.
We are grateful to Alejandra Garrido for discussions on the EGS groups
and for pointing out a lack of detail in the proof
of~\cite[Proposition~5.2]{Paper}.  Finally we thank the anonymous
referee for their detailed comments which led to
significant improvements in the exposition of our paper.


\end{document}